\documentclass[11pt,reqno]{amsart}

\date{May 16, 2013}

\usepackage[ansinew]{inputenc}
\usepackage{textcomp}
\usepackage{cite}

\usepackage{amsthm}
\usepackage{amssymb}
\usepackage{amsfonts}

\newtheorem{theorem}{Theorem}
\newtheorem{proposition}[theorem]{Proposition}
\newtheorem{lemma}[theorem]{Lemma}
\newtheorem{corollary}[theorem]{Corollary}

\theoremstyle{definition}

\theoremstyle{remark}
\newtheorem{remark}{Remark}

\numberwithin{equation}{section}

\newcommand{\lk}{\left(}
\newcommand{\rk}{\right)}

\newcommand{\tr}{\textnormal{Tr}}

\newcommand{\N}{\mathbb{N}}
\newcommand{\R}{\mathbb{R}}
\newcommand{\Sph}{\mathbb{S}}
\newcommand{\Le}{L^{(1)}_{s,d}}
\newcommand{\Lz}{L^{(2)}_{s,d}}
\newcommand{\Lzt}{\tilde L^{(2)}_{s,d}}

\DeclareMathOperator{\ran}{ran}
\DeclareMathOperator{\re}{Re}

\begin{document}

\title[Refined Semiclassical Asymptotics]{Refined Semiclassical Asymptotics for Fractional Powers of the Laplace Operator}

\author{Rupert L. Frank}
\address{Rupert L. Frank, Department of Mathematics, Princeton University, Washington Road, Princeton, NJ 08544, USA; Current address: Mathematics 253-37, Caltech, Pasadena, CA 91125, USA}
\email{rlfrank@math.princeton.edu, rlfrank@caltech.edu}

\author{Leander Geisinger}
\address{Leander Geisinger, Department of Mathematics, Universit\"at Stuttgart, Pfaffenwaldring 57, 70569 Stuttgart, Germany; Current address: Department of Mathematics, Princeton University, Washington Road, Princeton, NJ 08544, USA}
\email{leander.geisinger@mathematik.uni-stuttgart.de, leander@princeton.edu}

\thanks{\copyright\, 2013 by the authors. This paper may be reproduced, in its entirety, for non-commercial purposes.}

\begin{abstract}
We consider the fractional Laplacian on a domain and investigate the asymptotic behavior of its eigenvalues. Extending methods from semi-classical analysis we are able to prove a two-term formula for the sum of eigenvalues with the leading (Weyl) term given by the volume and the subleading term by the surface area. Our result is valid under very weak assumptions on the regularity of the boundary.
\end{abstract}


\maketitle

\section{Introduction and main result}
\label{sec:int}


\subsection{Introduction}

In this paper we study the asymptotic behavior of eigenvalues for fractional
powers of the Laplacian. The operator $(-\Delta)^s$ with $0<s<1$ appears in
numerous fields of mathematical physics, mathematical biology and mathematical finance and
has attracted a lot of attention recently. The key difference between this
operator and the usual Laplacian is the non-locality of $(-\Delta)^s$, which
allows one to model long-range interactions in applications and leads to
challenging mathematical problems.

{F}rom a probabilistic point of view, the fractional Laplacian of order $s$ on a
domain $\Omega\subset\R^d$ can be defined as the generator of the $2s$-stable
process killed upon exiting $\Omega$. A more operator theoretic definition,
which we employ here, is in terms of the quadratic form
\begin{equation}\label{eq:quadform}
 C_{s,d} \, \int_{\R^d} \int_{\R^d} \frac{|u(x)- u(y)|^2}{|x-y|^{d+2s}} dx
\, dy  
= \int_{\R^d} |p|^{2s} | \hat u(p) |^2 \, dp \,,
\end{equation}
restricted to functions $u\in H^s(\R^d)$ which satisfy $u\equiv 0$ in
$\R^d\setminus\overline\Omega$. Here $H^s(\R^d)$ is the Sobolev space of order
$s$, $\hat u(p) = (2\pi)^{-d/2} \int e^{-ip\cdot x} u(x)\,dx$ is the Fourier
transform of $u$ and $C_{s,d}$ is an explicit constant
given in \eqref{eq:csd}. The identity in \eqref{eq:quadform} is an easy consequence of Plancherel's theorem.

For bounded domains $\Omega$ the spectrum of the fractional
Laplacian is discrete and we denote its eigenvalues (in increasing order,
repeated according to multiplicities) by $\lambda_n^{(s)}$. Our main result in
this paper is a two-term asymptotic expansion of the sum of these eigenvalues,
\begin{equation}
 \label{eq:mainintro}
\frac1N \sum_{n=1}^N \lambda_n^{(s)} = C_{d,s}^{(1)} |\Omega|^{-2s/d}\ N^{2s/d}
+ C_{d,s}^{(2)} |\partial\Omega| |\Omega|^{-(d-1+2s)/d}\ N^{(2s-1)/d}\ (1+o(1))
\quad \mathrm{as}\ N\to\infty \,.
\end{equation}
Here $|\Omega|$ and $|\partial\Omega|$ denote the $d$-dimensional measure of
$\Omega$ and the $(d-1)$-dimensional surface measure of $\partial\Omega$,
respectively, and $C_{d,s}^{(1)}$ and $C_{d,s}^{(2)}$ are positive, universal
constants, depending only on $d$ and $s$, for which we shall obtain explicit expressions. Our result is valid for non-smooth
domains, requiring only that $\partial\Omega\in C^{1,\alpha}$ for some
(arbitrarily small) $\alpha>0$. It is remarkable that, despite the fact that we
are dealing with a non-local operator, both coefficients in \eqref{eq:mainintro}
have a local form, depending only on $\Omega$ and $\partial\Omega$, just like
in the case of the Laplacian. This will become clearer from the reformulation
given in Theorem \ref{thm:main} below.

In order to avoid confusion, we emphasize that the fractional Laplacian of order $s$ on a domain $\Omega$ is different from
the Dirichlet Laplacian on $\Omega$ raised to the $s$-th power. For the
Dirichlet Laplacian, and hence for its fractional powers, asymptotics analogous
to \eqref{eq:mainintro} are well-known. One of our results is that, while the
first terms in \eqref{eq:mainintro} coincide for both operators, the second
terms do \emph{not}. This means, in particular, that our result cannot be
obtained from the study of the (local) Dirichlet Laplacian, and that our
analysis needs to take into account the non-locality inherent in
\eqref{eq:mainintro}. For further results about the relation between the
fractional Laplacain on a domain and the fractional power of the Dirichlet
Laplacian we refer to \cite{CheSon05}; see also Section \ref{sec:snd} below.

The one-term asymptotics $\lambda_N^{(s)} = \frac{d+2s}d \ C_{d,s}^{(1)}
|\Omega|^{-2s/d}\ N^{2s/d}(1+o(1))$, which is a fractional version of Weyl's
law, is a classical result of Blumenthal and Getoor \cite{BluGet59}. More
recently, Ba{\~n}uelos and Kulczycki \cite{BanKul08} and Ba{\~n}uelos, Kulczycki
and Siudeja \cite{BaKuSi09} have shown a two-term asymptotic formula for
$\sum_{n=1}^\infty \exp(-t\lambda_n^{(s)})$ as $t\to0$. Note that
$\sum_{n=1}^\infty \exp(-t\lambda_n^{(s)})$ and $N^{-1} \sum_{n=1}^N
\lambda_n^{(s)}$ correspond to the Abel and Ces\`aro summation of the sequence
$\lambda_n^{(s)}$, respectively. As is well-known, asymptotics of Ces\`aro means
imply asymptotics of Abel means, but not vice versa. Hence for $C^{1,\alpha}$
domains we recover and improve upon the result of \cite{BanKul08,BaKuSi09}.

This is, actually, a significant improvement since our asymptotics are no longer
derived for the infinitely smooth function $e^{-tE}$ of the fractional
Laplacian, but, as we shall see shortly, for the Lipschitz function
$(\Lambda-E)_+$.
Moreover, since we are no longer able to apply the probabilistic machinery
available for the partition function, we have to find new and
more robust tools. Our methods also work for the ordinary Dirichlet
Laplacian on a bounded domain, and in \cite{FraGei11a} we use the techniques
developed here to give an elementary and short proof of two-term asymptotics in
that case.

Another point in which we go beyond \cite{BanKul08,BaKuSi09} is that we give an expression for the constant $C_{d,s}^{(2)}$ in \eqref{eq:mainintro} in terms of a model operator on a \emph{half-line} instead of a model operator on a \emph{half-space}. In this way our expression is similar to familiar two-term formulas in semi-classical analysis; see, for instance, \cite{SaVa}. This is possible due to some recent beautiful results of Kwa\'snicki \cite{Kwasni10a} about a general class of half-line operators.

We find it convenient to prove \eqref{eq:mainintro} in an equivalent form,
namely,
\begin{equation}
 \label{eq:mainintroramp}
\sum_{n=1}^\infty \lk \Lambda - \lambda_n^{(s)} \rk_+ 
=  \Le \, |\Omega| \, \Lambda^{1+d/2s} - \Lz \, |\partial \Omega|
\, \Lambda^{1+(d-1)/2s} (1+o(1))
\quad \mathrm{as}\ \Lambda\to\infty \,.
\end{equation}
Here $x_+:=\max\{x,0\}$ denotes the positive part of a number $x$. (The fact
that \eqref{eq:mainintro} and \eqref{eq:mainintroramp} are equivalent is
well-known to experts in the field, but we include a short proof in the appendix
for the sake of completeness, see Lemma \ref{asymptequiv}.) Note also that
\eqref{eq:mainintroramp} can be rewritten as
\begin{equation}
 \label{eq:mainintrosc}
\sum_{n=1}^\infty \lk 1 - h^{2s} \lambda_n^{(s)} \rk_+
=  \Le \, |\Omega| \, h^{-d} - \Lz \, |\partial \Omega|
\, h^{-d+1} (1+o(1))
\quad \mathrm{as}\ h\to0+ \,,
\end{equation}
and this is the form in which we shall state and prove our main theorem. The
small parameter $h$ has the interpretation of Planck's constant and
\eqref{eq:mainintrosc} emphasizes the semi-classical nature of the problem.

Our approach extends the multiscale analysis to the fractional setting. By this
we mean that we localize simultaneously on different length scales according to
the distance from the boundary. Of
course, a main difficulty when dealing with our non-local operator comes from
the treatment of the localization error. At this point we have to improve upon
previous results from \cite{LieYau88,SoSoSp10}. Another major impass, as
compared
to the local case, is the analysis of a one-dimensional model operator for which
an (almost) explicit diagonalization is far from trivial. This is where
Kwa\'snicki's work \cite{Kwasni10a} enters. It requires, however, still
substantial work to bring these results into a form which is useful for us.
We will explain the strategy of
our proof in more detail in Subsection \ref{ssec:strategy} after a precise
statement of our main result.

Throughout this paper we assume that the dimension $d\geq 2$. In the
one-dimensional case (the fractional Laplacian on an interval) considerably
stronger results are known \cite{KKMS10,Kwasni10b}. The powerful methods
developed there are, however, intrinsically one-dimensional and seem of little
help in the multi-dimensional case. The question raised in \cite{BaKuSi09} of
whether an analogue of Ivrii's two-term asymptotics \cite{Ivrii80a} holds for
$\lambda_n^{(s)}$ in $d\geq 2$ without Abel or Ces\`aro averaging remains a
challenging open problem.


\subsection{Main Result}
\label{ssec:res}

Let $\Omega\subset\R^d$, $d\geq 2$, be a bounded open set.  For $h > 0$ and $0 < s < 1$ let
$$
H_\Omega = (-h^2 \Delta)^s - 1
$$
be the self-adjoint operator in $L^2(\Omega)$ generated by the quadratic form
$$
\lk u , H_\Omega u \rk \, = \,  \int_{\R^d} \lk |hp|^{2s}-1 \rk | \hat u(p) |^2 \, dp 
$$
with form domain
$$
\mathcal{H}^s(\Omega) \, = \, \left\{ u \in H^s(\R^d) \, : \, u \equiv 0 \ \textnormal{on} \ \R^d \setminus \overline \Omega \right\} \, .
$$
For $0 < s < 1$ we have the representation
$$
\lk u , H_\Omega u \rk \, = \,  C_{s,d} \, h^{2s} \int_{\R^d} \int_{\R^d} \frac{|u(x)- u(y)|^2}{|x-y|^{d+2s}} dx \, dy  - \int_\Omega |u(x)|^2 \, dx \, 
$$
with constant
\begin{equation}
\label{eq:csd}
C_{s,d} \, = \, 2^{2s-1} \pi^{-d/2} \frac{\Gamma(d/2+s)}{|\Gamma(-s)|} \, > \, 0
\, .
\end{equation}

Our main results hold without any global geometric conditions on $\Omega$. We
only require weak  smoothness conditions on the boundary - namely that the
boundary belongs to the class $C^{1,\alpha}$ for some $\alpha > 0$. That is, the
local charts of $\partial \Omega$ are differentiable and the
derivatives are H\"older continuous with exponent $\alpha$.

\begin{theorem}
\label{thm:main}
Let $0<s<1$ and assume that the boundary of $\Omega$ satisfies $\partial \Omega
\in C^{1,\alpha}$ with some $0 < \alpha \leq 1$. Then
\begin{equation}
\label{eq:main}
\tr (H_\Omega)_- \, =  \, \Le \, |\Omega| \, h^{-d} - \Lz \, |\partial \Omega|
\, h^{-d+1} + R_h
\end{equation}
with $R_h=o(h^{-d+1})$ as $h\to 0+$. Here
\begin{equation}
\label{eq:weylconst}
\Le \, = \, \frac{1}{(2\pi)^d} \int_{\R^d} \lk |p|^
{2s}-1 \rk_- \, dp
\end{equation}
and the positive constant $\Lz$ is given in (\ref{eq:mod:lz}).

More precisely, we have the lower bound $R_h \geq -C h^{-d+1+\epsilon_-}$ for
any
$$
0< \epsilon_- \, < \, \left\{ \begin{array}{ll} 
\frac{\alpha}{\alpha+2} & \textnormal{if} \ 1/2 \leq s < 1 \,, \\
\frac{2s\alpha}{\alpha+1+2s} & \textnormal{if} \ 0 < s < 1/2 \,,
\end{array} \right.
$$
and the upper bound $R_h \leq C h^{-d+1+\epsilon_+}$ for
any
\begin{align*}
& 0< \epsilon_+ < \frac{\alpha}{\alpha+2} \qquad \textnormal{if} \ 1-d/4 \leq s
< 1 \,, \\
& 0< \epsilon_+ \leq \frac{\alpha(2s-1+d/2)}{\alpha+2s+d/2}  \qquad
\textnormal{if} \ 0< s < 1-d/4 \,.
\end{align*}
\end{theorem}

We do not claim that our remainder estimates are sharp. They show, however, that our methods are rather explicit and they correctly reflect the intuitive fact that the estimate worsens as the boundary gets rougher. We also mention that for not too small $s$ we (almost) get the same remainder estimate $h^{-d+1+\alpha/(\alpha+2)}$ that our method yields in the local case $s=1$ \cite{FraGei11a}.

In Section \ref{sec:snd} we will derive several representations of the constant $\Lz$ in \eqref{eq:main}. One of these, which emphasizes the semi-classical nature of the problem, leads to a rewriting of \eqref{eq:main} as
\begin{equation}
\label{eq:main2}
\tr (H_\Omega)_-  =   \iint_{T^*\Omega} \lk |p|^
{2s}-1 \rk_- \frac{dp dx}{(2\pi h)^d}  - \iint_{T^*\partial \Omega}
\zeta(|p'|^{-2s}) \frac{dp' d\sigma(x)}{(2\pi h)^{d-1}} + R_h \, ,
\end{equation}
where $T^*\Omega=\Omega\times\R^d$ and
$T^*\partial\Omega=\partial\Omega\times\R^{d-1}$ are the cotangent bundles over
$\Omega$ and $\partial\Omega$, respectively, and where $d\sigma$ is the surface
element of $\partial \Omega$. Here $\zeta$ is a universal (i.e., depending on
$s$, but independent of $\Omega$ or $d$) function, which has the interpretation
of an energy shift (the integral of a spectral shift). It is given in terms of a
one-dimensional model operator $A^+$ on the half-line $\R_+$ and its analogue
$A$ on the whole line (see Section \ref{sec:half}) by
$$
\zeta(\mu) \, = \, \mu^{-1} \int_0^\infty \lk a(t,t,\mu) - a^+(t,t,\mu) \rk dt \, , \quad \mu > 0 \, ,
$$
where $a(t,u,\mu)$ and $a^+(t,u,\mu)$ denote the integral kernels of $(A-\mu)_-$
and $(A^+-\mu)_-$, respectively. Another representation, derived in Remark
\ref{rephp}, shows that our result is consistent with the result of
\cite{BanKul08,BaKuSi09}.

In Section \ref{sec:snd} we also prove that
$$
\Lz>0 \,.
$$
Moreover, we compare this constant with the one obtained from the corresponding fractional power of the Dirichlet Laplacian.

\begin{proposition}\label{comp}
Let $0<s<1$ and assume that the boundary of $\Omega$ satisfies $\partial \Omega \in C^{1,\alpha}$ with some $0 < \alpha \leq 1$. Let $-\Delta_\Omega$ be the Dirichlet Laplacian on $\Omega$. Then
\begin{equation}
\label{eq:maincomp}
\tr\left( \left(-h^2\Delta_\Omega\right)^s -1 \right)_- \, =  \, \Le \, |\Omega| \, h^{-d} - \Lzt \, |\partial \Omega|
\, h^{-d+1} + R_h
\end{equation}
with $R_h=o(h^{-d+1})$ as $h\to 0+$. Here $\Le$ is the same as in \eqref{eq:weylconst} and $\Lzt$ satisfies
\begin{equation}
\label{eq:compconst}
\Lz < \Lzt \,.
\end{equation}
\end{proposition}

In other words, the operators $H_\Omega$ and $\left(-h^2\Delta_\Omega\right)^s -1$ differ semi-classically to first subleading order.

\subsection{Strategy of the proof}
\label{ssec:strategy}

The proof of Theorem \ref{thm:main} is divided into
three main steps: First, we localize the operator $H_\Omega$ into balls, whose
size varies depending on the distance to the complement of $\Omega$. Then we can
analyze separately the semiclassical limit in the bulk and at the boundary.

The key idea is to choose the localization depending on the distance to the complement of $\Omega$, see \cite[Theorem 17.1.3]{Hoerma85} and \cite{SolSpi03}. 
Let $d(u) = \inf \{|x-u| \, : \,  \, x \notin \Omega  \}$ denote the distance of $u \in \R^d$ to the complement of $\Omega$. We set
\begin{equation}
\label{eq:l}
l(u) \, = \, \frac 12 \lk 1 + \lk d(u)^2 + l_0^2 \rk^{-1/2} \rk^{-1} \, , 
\end{equation}
where  $0 < l_0 \leq 1/2$ is a small parameter depending only on $h$. Indeed, we
will finally choose $l_0$ proportional to $h^\beta$ with suitable $0 < \beta <
1$.

In Section \ref{sec:loc} we construct real-valued functions $\phi_u \in C_0^\infty(\R^d)$ with support in the ball $B_u = \{ x \in \R^d \, : \, |x-u| < l(u) \}$. For all $u \in \R^d$ these functions satisfy
\begin{equation}
\label{eq:int:grad}
\left\| \phi_u \right\|_\infty \, \leq \, C \ ,  \qquad \left\| \nabla \phi_u \right\|_\infty \leq C \, l(u)^{-1}
\end{equation}
and for all $x \in \R^d$
\begin{equation}
\label{eq:int:unity}
\int_{\R^d} \phi_u^2(x) \, l(u)^{-d} \, du \, = \, 1 \, .
\end{equation}
Here and in the following the letter $C$ denotes various positive constants  that are independent of $u$, $l_0$ and $h$. 

\begin{proposition}
\label{pro:loc}
There is a constant $C > 0$ depending only on $s$ and $d$ such that for all $0 < l_0 \leq 1/2$ and all $0 < h \leq C^{-1} l_0$ the estimates
$$
0 \, \leq \,  \tr (H_\Omega)_- -  \int_{\R^d} \textnormal{Tr} \lk \phi_u
H_\Omega \phi_u \rk_- l(u)^{-d} \, du  \, 
\leq \, C \, h^{-d+2} \, l_0^{-1} R_{\textnormal{loc}}(l_0,h)
$$
hold with a remainder
$$
 R_{\textnormal{loc}}(l_0,h) \, = \, \left\{ \begin{array}{ll} 1 &
\textnormal{if} \ 1-d/4 < s < 1 \\
|\ln (l_0/h)|^{1/2} &  \textnormal{if} \ 0 < s =1-d/4  \\
(l_0/h)^{2-2s-d/2} & \textnormal{if} \  0< s < 1-d/4
\end{array} \right. \, .   
$$
\end{proposition}

In view of this result, one can analyze the local asymptotics, i.e.,  the
asymptotic behavior of $\tr(\phi_u H_\Omega \phi_u)_-$, separately on different
parts of $\Omega$. First, we consider the bulk, where the influence of the
boundary is not felt.

\begin{proposition}
\label{pro:bulk}
Assume that $\phi \in C_0^1(\Omega)$ is real-valued, supported in a ball of radius $l > 0$ and
\begin{equation}
\label{eq:int:gradphi}
\| \nabla \phi \|_\infty \, \leq \, C  l^{-1} \, .
\end{equation}
Then for all $h > 0$ the estimates
$$
- C  l^{d-2}  h^{-d+2} \, \leq \, \tr \lk \phi H_{\Omega} \phi \rk_- - \Le \int_\Omega \phi^2(x) \, dx \, h^{-d}   \, \leq \,   0
$$
hold with a constant depending only on the constant in (\ref{eq:int:gradphi}).
\end{proposition}

Close to the boundary of $\Omega$, more precisely, if the support of $\phi$ intersects the boundary, a boundary term of the order $h^{-d+1}$ appears.

\begin{proposition}
\label{pro:boundary}
There is a constant $c > 0$ depending only on $\Omega$ such that the following holds. Assume that $\phi \in C_0^1(\R^d)$ is real-valued, supported in a ball of radius $0<l \leq c$ intersecting the boundary of $\Omega$ and satisfies (\ref{eq:int:gradphi}). Then for all $h > 0$ the estimates
$$
-\tilde R_{\textnormal{bd}}(l,h)  \, \leq \, \tr \lk \phi H_{\Omega} \phi \rk_- - \Le \int_\Omega \phi^2(x) dx  h^{-d} + \Lz \int_{\partial \Omega} \phi^2(x) d\sigma(x) h^{-d+1} \, \leq \, R_{\textnormal{bd}}(l,h)
$$
hold. Here $d\sigma$ denotes the $(d-1)$-dimensional volume element of $\partial \Omega$ and the remainder terms satisfy for any $0 < \delta_1 < 1$ and $0 < \delta_2 < \min\{ 1,2s \}$
\begin{align*}
R_{\textnormal{bd}}(l,h) \, &\leq \, C_{\delta_1} \lk
\frac{l^{d-1-\delta_1}}{h^{d-1-\delta_1}} +
\frac{l^{d+\alpha}}{h^d} \rk \, ,\\
\tilde R_{\textnormal{bd}}(l,h) \, &\leq \, C_{\delta_1,\delta_2} \lk  \frac{l^{d-1-\delta_1}}{h^{d-1-\delta_1}} + \frac{l^{d-1-\delta_2}}{h^{d-1-\delta_2}} + \frac{l^{2\alpha+d-1}}{h^{d-1}} + \frac{l^{d+\alpha}}{h^d} \rk \, ,
\end{align*}
with positive constants $C_{\delta_1}$, $C_{\delta_1,\delta_2}$ depending on $\delta_1$, $\delta_2$, $\Omega$ and the constant in  (\ref{eq:int:gradphi}).
\end{proposition}

Based on these propositions we can complete the proof of Theorem \ref{thm:main}.

\begin{proof}[Proof of Theorem \ref{thm:main}]
In order to apply Proposition \ref{pro:boundary} to the operators $\phi_u H_\Omega \phi_u$, we need to estimate $l(u)$ uniformly. Let
$$
U(\Omega) \, = \,\{ u \in \R^d \, : \, B_u \cap \partial \Omega \neq \emptyset \}
$$ 
be a small neighborhood of the boundary.
For $u \in U(\Omega)$ we have $d(u) \leq l(u)$, which by the definition of $l(u)$ implies
\begin{equation}
\label{eq:int:luup}
l(u) \leq \,  l_0/ \sqrt3 \, .
\end{equation}

In view of (\ref{eq:int:grad}) and (\ref{eq:int:luup}) we can apply Proposition \ref{pro:bulk} and Proposition \ref{pro:boundary} to all functions $\phi_u$, $u \in \R^d$, if $l_0$ is sufficiently small. Combining these results with Proposition \ref{pro:loc} we get
\begin{align*}
& -  C \int_{\Omega \setminus U(\Omega)} l(u)^{-2}  du \, h^{-d+2}  -  \int_{U(\Omega)}  \tilde R_{\textnormal{bd}}(l(u),h) \, l(u)^{-d}  du \\
& \leq \,  \tr \lk H_{\Omega} \rk_- - \Le \int_{\R^d} \int_\Omega \phi_u^2(x) dx \, \frac{du}{l(u)^d} \, h^{-d} + \Lz \int_{\R^d} \int_{\partial \Omega} \phi_u^2(x) d\sigma(x) \, \frac{du}{l(u)^d} \, h^{-d+1} \\
& \leq \,  \int_{U(\Omega)}  R_{\textnormal{bd}}(l(u),h) l(u)^{-d}  du + C
h^{-d+2} l_0^{-1} R_{\textnormal{loc}}(l_0,h) \, .
\end{align*}
Now we change the order of integration and in view of (\ref{eq:int:unity}) we obtain
\begin{align}
\nonumber
& - C \int_{\Omega \setminus U(\Omega)} l(u)^{-2}  du \, h^{-d+2} - 
\int_{U(\Omega)} \tilde R_{\textnormal{bd}}(l(u),h) \, l(u)^{-d} du \\
\nonumber
& \leq \,  \tr \lk H_{\Omega} \rk_- - \Le \, |\Omega| \, h^{-d} + \Lz \, |\partial \Omega| \, h^{-d+1}   \\
\label{eq:int:rem1}
& \leq \,  \int_{U(\Omega)} R_{\textnormal{bd}}(l(u),h) l(u)^{-d}  du +  C
h^{-d+2} l_0^{-1} R_{\textnormal{loc}}(l_0,h) \, .
\end{align}
It remains to estimate the error terms.

By definition of $l(u)$ we have
\begin{equation}
\label{eq:int:lulo}
l(u) \geq \,  \frac 14 \min \lk d(u), 1 \rk \quad \textnormal{and} \quad l(u) \, \geq \frac{l_0}{4} 
\end{equation}
for all $u \in \R^d$ and $l_0 \leq 1$. For $u \in \Omega \setminus U(\Omega)$, we find $d(u) \geq l(u) \geq l_0/4$. Hence, we can estimate
$$
\int_{\Omega \setminus U(\Omega)} l(u)^{-2} du \, \leq \, C \lk 1 + \int_{\{d(u) \geq l_0/4 \} } d(u)^{-2} du \rk  \leq C \lk 1 + \int_{l_0/4}^\infty t^{-2} \, |\partial \Omega_t | \, dt \rk \, ,
$$
where $|\partial \Omega_t|$ denotes the surface area of the boundary of $\Omega_t = \{ x \in \Omega \, : \, d(x) > t \}$. Using the fact that $|\partial \Omega_t|$ is uniformly bounded and that $|\partial \Omega_t| = 0$ for large $t$, we get
\begin{equation}
\label{eq:int:U1}
\int_{\Omega \setminus U(\Omega)} l(u)^{-2} du \, \leq \, C l_0^{-1} \, .
\end{equation}
For $u \in U(\Omega)$ the inequalities (\ref{eq:int:luup}) and (\ref{eq:int:lulo}) show that $l(u)$ is proportional to $l_0$. Since $B_u \cap \partial \Omega \neq \emptyset$ we find $d(u) \leq l(u) \leq l_0$ and 
\begin{equation}
\label{eq:int:U2}
\int_{U(\Omega)} l(u)^a du \, \leq \, C l_0^a \int_{\{ d(u) \leq l_0 \} } du \, \leq \, Cl_0^{a+1} \, ,
\end{equation}
for any $a \in \R$.

We insert (\ref{eq:int:U1}) and (\ref{eq:int:U2}) into (\ref{eq:int:rem1}) and
get (using the fact that $h\leq C^{-1} l_0$)
\begin{align}
\nonumber
- C \lk l_0^{-\delta_2} h^{\delta_2} + 
l_0^{2\alpha} +  l_0^{\alpha+1} h^{-1} \rk  
& \leq \,  h^{d-1} \left( \tr \lk H_{\Omega} \rk_- - \Le \, |\Omega| \, h^{-d} +
\Lz \, |\partial \Omega| \, h^{-d+1} \right) \\
\label{eq:int:est1}
& \leq \, C \lk  l_0^{-\delta_1} h^{\delta_1} + l_0^{\alpha+1} h^{-1} +
l_0^{-1} h
\,R_{\textnormal{loc}}(l_0,h)  \rk  \, .
\end{align}

In order to choose $l_0$ we need to distinguish several cases. For the lower
bound we recall that $0<\delta_2<\min\{1,2s\}$. The stated lower bound on $R_h$
follows with $l_0$ proportional to $h^\beta$, where $\beta =
(1+\delta_2)/(1+\alpha+\delta_2)$. 

For the upper bound we have $0<\delta_1<1$. If $1-d/4<s<1$, we pick
$l_0$ proportional to $h^\beta$, where $\beta =(1+\delta_1)/(1+\alpha+\delta_1)
$. If $0<s\leq 1-d/4$, we pick $h^\beta$, where $\beta
=(2s+d/2)/(\alpha+2s+d/2)$. This completes the proof of Theorem \ref{thm:main}.
\end{proof}

The remainder of the text is structured as follows. First we 
analyze the local asymptotics in the bulk and prove Proposition \ref{pro:bulk}. 
This is done in Section \ref{sec:bulk}. In Section~\ref{sec:half} we consider
the local asymptotics in the case where $\Omega$ is replaced by a half-space. We
reduce the problem close to the boundary to the analysis of a one-dimensional
model operator given on a half-line and give an analogue of
Proposition~\ref{pro:boundary} for a half-space.
In Section~\ref{sec:boundary} we show how Proposition~\ref{pro:boundary} follows
from the previous considerations by local straightening of the boundary.
In Section~\ref{sec:loc}, we perform
the localization and, in particular, prove Proposition \ref{pro:loc}.
In the appendix we provide some technical results about the one-dimensional model operator introduced in Section \ref{sec:half}.

\subsection*{Notation} We define the positive and negative parts of a real
number $x$ by $x_\pm=\!\max\{0,\pm x\}$. We use a similar notation for the Heaviside function, namely, $x_\pm^0 = 1$ if $\pm x \geq 0$ and $x_\pm^0 = 0$ if $\pm
x < 0$. For a self-adjoint operator $X$, the operators $X_\pm$ and $X_\pm^0$ are
defined similarly via the spectral theorem.


\section{Local asymptotics in the bulk}
\label{sec:bulk}

This section is a warm-up dealing with the spectral asymptotics in the boundaryless case. Although the estimates in this case are essentially known, we include a proof for the sake of completeness and in order to introduce the methods that will be important later on. We divide the proof of Proposition \ref{pro:bulk} into two subsections containing the lower and the upper bound, respectively. The operator 
$$
H_0 = (-h^2 \Delta)^s - 1 \qquad\mathrm{in}\ L^2(\R^d) \,,
$$
defined with form domain $H^s(\R^d)$, will appear frequently.


\subsection{Lower bound on $-\tr \lk \phi H_\Omega \phi \rk_-$}
\label{sec:bu:lo}

The lower bound is given by a variant of the Berezin--Lieb--Li--Yau inequality,
see \cite{Berezi72b,Lieb73,LiYau83}. For later purposes we record this as

\begin{lemma}
\label{lem:berezin}
For any $\phi \in L^2(\R^d)$ and $h > 0$
$$
\tr \lk \phi H_\Omega \phi \rk_- \, \leq \, \Le \int_{\R^d} \phi^2(x) \, dx \,
h^{-d} \, .
$$
\end{lemma}

\begin{proof}
We apply the variational principle for the sum of the eigenvalues
$$
-\tr \lk \phi H_\Omega \phi \rk_- \, = \, \inf_{0 \leq \gamma \leq 1} \tr  \lk \gamma \phi H_\Omega \phi \rk \, ,
$$
where the infimum is taken over all trial density matrices, i.e., over all trace-class operators $0 \leq \gamma \leq 1$ with range belonging to the form domain of $H_\Omega$. We apply this twice and find
$$
\tr \lk \phi H_\Omega \phi \rk_- \, \leq \, \tr \lk \phi H_0 \phi \rk_- 
\, \leq \, \tr \lk \phi \lk H_0 \rk_- \phi \rk \,.
$$
Applying the Fourier transform to diagonalize the operator $(H_0)_-$ yields the
bound
$$
\tr \lk \phi \lk H_0 \rk_- \phi \rk \, = \, \frac{1}{(2\pi h)^d} \iint \phi(x)^2 \lk
|p|^{2s} -1 \rk_- \, dp \, dx \, = \, \Le \int \phi(x)^2 \, dx \, h^{-d} \, ,
$$
as claimed.
\end{proof}


\subsection{Upper bound on $-\tr \lk \phi H_\Omega \phi \rk_-$}

We now assume that $\phi$ satisfies the conditions of Proposition \ref{pro:bulk}. In particular, we assume that $\phi$ has support in $\Omega$. To derive the upper bound we put $\gamma = (H_0)_-^0$, i.e.,
$$
\gamma(x,y) \, = \, (2\pi h)^{-d} \, \int_{|p|<1}  e^{ip \cdot (x-y)/h} \, dp \,,
$$
and obtain that 
\begin{align}
\label{eq:bu:trail}
- \textnormal{Tr} \lk \phi H_\Omega \phi \rk_-  
\, & \leq \,  \textnormal{Tr} \lk \gamma \phi H_\Omega \phi \rk \, = \, \textnormal{Tr} \lk \gamma \phi H_0 \phi \rk \notag \\
\, & = \,  \int_{|p|<1} \lk \| (-h^2 \Delta )^{s/2} \phi \, e^{ip \cdot \ /h }
\|^2_2 -  \left\| \phi \right\|^2_2 \rk \frac{dp}{(2 \pi h)^d} \, .
\end{align}

\begin{lemma}
\label{lem:bu:up}
For $\phi \in C_0^\infty(\R^d)$ and $h > 0$ we have
$$
\| (-h^2 \Delta )^{s/2} \phi \, e^{-ip \cdot \ /h } \|^2_2 \, = \, |p|^{2s} \left\| \phi \right\|_2^2 + \int \lk \frac 12 \lk  |p+h\eta|^{2s}  + |p-h\eta|^{2s} \rk - |p|^{2s} \rk | \hat \phi ( \eta) |^2 d\eta \, .
$$
\end{lemma}

\begin{proof}
By Plancherel's Theorem we get
$$
\| (-h^2 \Delta )^{s/2} \phi \, e^{ip \cdot \ /h } \|^2_2 \, = \, (2\pi h)^{-d} \iiint |\xi|^{2s} \, \phi(x) \phi(y) \, e^{i(p-\xi)\cdot(x-y)/h}   dy d\xi  dx \, .
$$
Since $\phi \in C^\infty_0 (\R^d)$, we can use the fact that
$$
\iint \phi(x) \phi(y)  e^{i(p-\xi)\cdot(x-y)/h}  dx  dy  =  \lim_{\delta \to 0+} \iint e^{-\delta|x-y|^2} \phi(x) \phi(y)  e^{i(p-\xi)\cdot(x-y)/h} dx  dy 
$$
and since  $|\xi|^{2s} \iint \phi(x) \phi(y) e^{i(p-\xi)\cdot(x-y)/h}  dx dy$ is absolutely integrable as a function of $\xi \in \R^d$ we find 
\begin{align}
\nonumber
\lefteqn{ \| (-h^2 \Delta )^{s/2} \phi \, e^{-ip \cdot \ /h } \|^2_2} \\
\nonumber
& =  \lim_{\delta \to 0+} \iiint e^{-\delta|x-y|^2}  |\xi|^{2s} \phi(x) \phi(y) \, e^{i(p-\xi)\cdot(x-y)/h} \, \frac{ dy dx d\xi }{(2 \pi h)^d} \\
\label{eq:bu_planch}
& =  \lim_{\delta \to 0+}  \iiint e^{-\delta|x-y|^2}  |\xi|^{2s} 
\lk \phi^2(x) + \phi^2(y) - \left( \phi(x) - \phi(y) \right)^2 \rk e^{i(p-\xi)\cdot(x-y)/h} \, \frac{ dx  dy  d\xi}{2(2\pi h)^d} \, .
\end{align}
By symmetry in $x$ and $y$ the first two terms on the right side give
\begin{align*}
& \iiint e^{-\delta|x-y|^2} |\xi|^{2s}  \phi^2(x) \,   e^{i(p-\xi)\cdot(x-y)/h} \, \frac{dx  dy d\xi}{(2 \pi h)^d}  \\
& = \, \lk \frac{\pi}{\delta} \rk^{d/2}  \iiint e^{-|p-\xi|^2/(4\delta h^2)}  |\xi|^{2s}  \phi^2(x)  \frac{dx d\xi}{(2 \pi h)^d} \, .
\end{align*}
Now we can substitute $|q|^2 = |p-\xi|^2/(4\delta h^2)$ to get 
\begin{equation}
\label{eq:bu_first}
\lim_{\delta \to 0+}  \iiint e^{-\delta|x-y|^2}  |\xi|^{2s}  \lk \phi^2(x) + \phi^2(y) \rk e^{i(p-\xi)\cdot(x-y)/h} \frac{ dx  dy d\xi}{2(2\pi h)^d} \, = \, |p|^{2s} \int \phi^2(x) \, dx \, .
\end{equation}
We are left with calculating the third term on the right side of (\ref{eq:bu_planch}), namely
$$
\iiint e^{-\delta |z|^2} \, |\xi|^{2s}  \left( \phi(x) - \phi(x+z) \right)^2  e^{i(p-\xi)\cdot z/h} \, \frac{ dx  dz  d\xi}{2(2\pi h)^d}  \, .
$$
Again by Plancherel's Theorem we see that it equals
$$
\iiint e^{-\delta |z|^2} \, |\xi|^{2s}  \left| \hat \phi \lk \frac \eta h \rk \right|^2 \, \left|1 -  e^{-iz \cdot \eta /h} \right|^2  e^{i(p-\xi)\cdot z/h} \, \frac{ d\eta  dz  d\xi}{2(2\pi)^d h^{2d}} \, .
$$
We can write  
$$
\left|1 -  e^{-iz \cdot \eta /h} \right|^2   \, = \, 2 - e^{iz\cdot \eta/h}- e^{-iz\cdot \eta/h}
$$
and from the first summand we get
$$
\iiint e^{-\delta |z|^2}  |\xi|^{2s}  \left| \hat \phi \lk \frac \eta h \rk \right|^2 e^{i(p-\xi)\cdot z /h}  \frac{ d\eta  dz  d\xi}{(2\pi)^d h^{2d}}  =  \iint e^{-|q|^2}  |p+2h\sqrt \delta q|^{2s}  \left| \hat \phi \lk \frac \eta h \rk \right|^2 \frac{ d\eta  dq}{\pi^{d/2} h^{d}} \, . 
$$
In the same way we can treat the second and third summand and after taking the limit $\delta \to 0+$ we finally find 
\begin{align}
\nonumber
& \lim_{\delta \to 0+}  \iiint e^{-\delta|x-y|^2} \, |\xi|^{2s}  \left( \phi(x) - \phi(y) \right)^2  e^{i(p-\xi)\cdot(x-y)/h} \frac{ dx  dy  d\xi}{2(2\pi h)^d} \\
\label{eq:bu_second}
& \, = \, \frac{1}{h^d} \int \lk |p|^{2s} - \frac 12 \lk  |p+\eta|^{2s}  + |p-\eta|^{2s} \rk \rk \left| \hat \phi \lk \frac{\eta}{h}  \rk \right|^2 d\eta \, .
\end{align}
Hence, combining (\ref{eq:bu_planch}), (\ref{eq:bu_first}) and (\ref{eq:bu_second}) yields the claim.
\end{proof}

In view of identity (\ref{eq:bu:trail}) and Lemma \ref{lem:bu:up} we conclude
\begin{equation}
\label{eq:bu_traceres}
\textnormal{Tr} \lk \gamma \phi H_0 \phi \rk \, = \, (2 \pi h)^{-d} \int_{|p| < 1} \lk |p|^{2s} - 1 \rk dp \left\| \phi \right\|^2_2 + (2 \pi h)^{-d} \int_{|p|<1} R_h(p) \, dp
\end{equation}
with
\begin{equation*}
R_h(p) \, = \, \int  \lk \frac 12 \lk  |p+h\eta|^{2s}  + |p-h\eta|^{2s} \rk - |p|^{2s} \rk |\hat \phi (\eta)|^2 d\eta \, .
\end{equation*}
We proceed to estimate $R_h(p)$. Note that for any $a>0$
$$
\max_{|t|\leq a} \left( (a+t)^s + (a-t)^s \right) \, = \, 2 a^s \,.
$$
Taking $a=|p|^2+|\eta|^2$ and $t=2p\cdot\eta$ we deduce that
$$
\frac12 \left( |p+\eta|^{2s} + |p-\eta|^{2s} \right) - |p|^{2s} 
\, \leq \, (|p|^2+|\eta|^2)^s - |p|^{2s} \,.
$$
Next, for $0<s<1$ concavity implies that $(a+b)^s \leq a^s + s a^{s-1} b$ for $a,b>0$, from which we learn that
$$
(|p|^2+|\eta|^2)^s - |p|^{2s} \, \leq \, s \, |p|^{2(s-1)} \, |\eta|^2 \,.
$$
Hence, replacing $\eta$ with $h \eta$ and using (\ref{eq:int:gradphi}) we can estimate
$$
R_h(p) \, \leq \, s\int |p|^{-2+2s} |h\eta|^2 |\hat \phi (\eta) |^2 d\eta \, = \, s \, |p|^{-2+2s} \, h^2 \int |\nabla \phi|^2 dx \, \leq \, C  h^2 |p|^{-2+2s} l^{d-2}\, .
$$
Thus the upper bound follows from (\ref{eq:bu:trail}) and 
(\ref{eq:bu_traceres}).


\section{Asymptotics on the half-space}
\label{sec:half}


Our goal in this section is to prove the analogue of Proposition \ref{pro:boundary} in the case where $\Omega$ is the half-space $\R^d_+=\{(x',x_d):\ x_d>0\}$. We define the operator $H^+$ on $L^2(\R^d_+)$, in the same way as $H_\Omega$, with form domain 
$$
\mathcal{H}^s(\R^d_+) \, = \, \left\{ v \in H^s(\R^d) \, : \, v \equiv 0 \
\textnormal{on} \ \R^d \setminus \overline{\R^d_+} \right\} \, .
$$
We shall prove

\begin{proposition}
\label{lem:mod}
Assume that $\phi \in C_0^1(\R^d)$ is supported in a ball of radius $l > 0$ and
assume that (\ref{eq:int:gradphi}) is satisfied. Then for $h > 0$ and any $0 <
\delta_1 < 1$ and $0 < \delta_2 < \min\{1,2s\}$ we have
\begin{align*}
& -C_{\delta_1,\delta_2} \lk l^{d-1-\delta_1} h^{-d+1+\delta_1}+ l^{d-1-\delta_2} h^{-d+1+\delta_2} \rk \\
&\leq \, \tr \lk  \phi H^+  \phi \rk_- -   \Le \int_{\R^d_+} \phi^2(x) dx  h^{-d} + \Lz \int_{\R^{d-1}} \phi^2(x',0) dx'  h^{-d+1} \\
& \leq \, C_{\delta_1} l^{d-1-\delta_1} h^{-d+1+\delta_1} \, .
\end{align*}
\end{proposition} 

This result depends on a more or less explicit diagonalization of the operator $H^+$, which is far from obvious. This is accomplished in Subsections \ref{ssec:modelhl} and \ref{ssec:hshl}, relying crucially on recent results of Kwa\'snicki \cite{Kwasni10a} about non-local operators on a half-line. These results are collected and extended to our needs in the appendix.


\subsection{The model operator on the half-line}\label{ssec:modelhl}

In this subsection we collect some facts about the one-dimensional operator
$$
A^+ \, = \, \lk -\frac{d^2}{dt^2} + 1 \rk^s
$$
in $L^2(\R_+)$ with form domain $\mathcal{H}^s(\R_+)$, and about the corresponding operator
$A$ in $L^2(\R)$, defined analogously to $A^+$, but with form domain $H^s(\R)$.

For $\mu > 0$ and $t,u \in \R_+$, let $e^+(t,u,\mu)$ and $a^+(t,u,\mu)$ be the
integral kernels of $(A^+-\mu)_-^0$ and $(A^+-\mu)_-$, respectively. Similarly, we
define $a(t,u,\mu)$ via $(A-\mu)_-$. To simplify notation we abbreviate
$a^+(t,\mu) = a^+(t,t,\mu)$. We also note that $a(\mu)=a(t,t,\mu)$ is
independent of $t\in\R_+$. The inequality $A^+\geq 1$ implies that $a^+(t,u,\mu)=e^+(t,u,\mu)=0$ for $\mu< 1$ and similarly for $a(t,u,\mu)$ and $e(t,u,\mu)$.

The following two results about $e^+(t,\mu)$ and $a^+(t,\mu)$ are rather technical and we defer the proofs to Appendices \ref{ap:eest} and \ref{ap:aest}. The first one provides a rough a-priori bound on $e^+(t,u,\mu)$.

\begin{lemma}
 \label{eest}
For any $\mu > 0$ and $t,u \in \R_+$ one has $|e^+(t,u,\mu)| \leq C
\mu^{1/2s}$.
\end{lemma}

The second result in this subsection quantifies that $a^+(t,\mu)$ is close to $a(\mu)$ for large $t$.

\begin{lemma}
\label{aest}
For any $0 \leq \gamma < 1$ there is a constant $C_\gamma$ such that for all
$\mu\geq 1$,
\begin{equation}
 \label{eq:aest}
\int_0^\infty t^\gamma | a^+(t,\mu) -a(\mu) |\,dt \leq C_\gamma \, \mu \lk (\ln \mu )^2 + 1 \rk \,.
\end{equation}
In particular, the function
\begin{equation}
\label{eq:K}
K(t) \, = \, \frac{1}{(2\pi)^{d-1}} \int_{\R^{d-1}} |\xi'|^{1+2s} \lk
a(|\xi'|^{-2s})-a^+(t|\xi'|,|\xi'|^{-2s}) \rk d\xi' \,, \qquad t>0\,,
\end{equation}
satisfies for every $0\leq \gamma<1$
$$
\int_0^\infty t^\gamma \, |K(t)| \, dt \, < \, \infty \, .
$$
\end{lemma}

With this lemma at hand we can now define the constant $\Lz$ which appears in our main theorem by
\begin{equation}
\label{eq:mod:lz}
\Lz \, = \, \int_0^\infty K(t) \, dt \, .
\end{equation}
(This integral converges by Lemma \ref{aest}.) Expression \eqref{eq:mod:lz}
suffices for the proof of our main result. In Section \ref{sec:snd}, see also
\eqref{eq:ap:lz}, we will derive different representations for $\Lz$.


\subsection{Reduction from the half-space to the half-line}\label{ssec:hshl}

Our goal in this subsection is to write the spectral projections of the
operator $H^+$ on the half-space in terms of those of the operator $A^+$ on the
half-line. Since $H^+$ commutes with translations parallel to the boundary of $\R^d_+$, it can be written as a direct integral; see, e.g., \cite[Sec. XIII.16]{ReSi} for definitions and properties of direct integrals.

\begin{lemma}
\label{lem:bdlo}
The mapping
$$
\left(\mathcal U f\right)_{\xi'} (t) = (2\pi h)^{-(d-1)/2} h^{1/2} |\xi'|^{-1/2} \int_{\R^{d-1}} f(x',|\xi'|^{-1} h t) e^{-i\xi'\cdot x'/h} \,dx' \,,
\quad \xi'\in\R^{d-1}\,,\ t>0 \,,
$$
defines a unitary operator from $L^2(\R^d_+)$ to $\int_{\R^{d-1}}^\oplus L^2(0,\infty) \,d\xi'$. Moreover,
$$
\mathcal U \left( H^++1\right) \mathcal U^* = \int_{\R^{d-1}}^\oplus |\xi'|^{2s} A^+ \,d\xi' \,.
$$
\end{lemma}

Before giving the proof we show how to deduce formulas for spectral projections.

\begin{corollary}\label{cor:bdlo}
For $x = (x',x_d) \in \R^d_+$ and $y = (y',y_d) \in \R^d_+$ the integral kernels of $(H^+)_-^0$ and $(H^+)_-$ are related to those of $(A^+-\mu)_-^0$ and $(A^+-\mu)_-$ by
\begin{equation}
 \label{eq:projh+}
(H^+)_-^0(x,y) = \frac{1}{h^d} \int_{\R^{d-1}} |\xi'|
e^{i \xi'\cdot(x'-y') /h} \, e^+ \lk \frac{x_d |\xi'|} h , \frac{y_d
|\xi'|} h,\frac1{|\xi'|^{2s}} \rk \frac{d\xi'}{(2\pi)^{d-1}}
\end{equation}
and
\begin{equation}
 \label{eq:densh+}
(H^+)_-(x,y) = \frac{1}{h^d} \int_{\R^{d-1}}
|\xi'|^{1+2s} \, e^{i\xi'\cdot (x'-y')/h} \, a^+ \lk \frac{x_d |\xi'|}h,
\frac{y_d |\xi'|}h, \frac1{|\xi'|^{2s}} \rk \frac{d\xi'}{(2\pi)^{d-1}} \, .
\end{equation}
\end{corollary}

\begin{proof}[Proof of Corollary \ref{cor:bdlo}]
Lemma \ref{lem:bdlo} and the spectral theorem imply that for any bounded, measurable function $\phi$ on $\R$,
$$
\mathcal U \, \phi(H^++1) \, \mathcal U^* = \int_{\R^{d-1}}^\oplus \phi(|\xi'|^{2s} A^+) \,d\xi' \,.
$$
This formula means that for any $f\in L^2(\R^d_+)$,
$$
(f,\phi(H^++1) f) = \int_{\R^{d-1}} \left( \left( \mathcal U f\right)_{\xi'} , \phi(|\xi'|^{2s} A^+) \left( \mathcal U f\right)_{\xi'} \right) \,d\xi' \,.
$$
From this, we easily conclude that if $\phi(|\xi'|^{2s} A^+)$ has an integral kernel for all $\xi'\in\R^{d-1}$, then $\phi(H^++1)$ has an integral kernel given by
$$
\phi(H^++1)(x,y) = \frac{1}{h^d} \int_{\R^{d-1}} |\xi'|
e^{i \xi'\cdot(x'-y') /h} \, \phi(|\xi'|^{2s} A^+) \lk \frac{x_d |\xi'|} h , \frac{y_d
|\xi'|} h \rk \frac{d\xi'}{(2\pi)^{d-1}} \,.
$$
The corollary now follows from the fact that for $\phi(E)=(E-1)_-^0$ one has $\phi(|\xi'|^{2s} A^+) = (A^+ - |\xi'|^{-2s})_-^0$ and for $\phi(E)=(E-1)_-$ one has $\phi(|\xi'|^{2s} A^+) = |\xi'|^{2s} (A^+ - |\xi'|^{-2s})_-$.
\end{proof}

We now give the

\begin{proof}[Proof of Lemma \ref{lem:bdlo}]
The fact that $\mathcal U$ is unitary follows from Plancherel's theorem together with a dilation. To prove the formula for $H^+$, let $f\in \mathcal{H}^s(\R^d_+)$, the form domain of $H^+$, and denote by $\hat f$ as before the Fourier transform of $f$ with respect to both $x'$ and $x_d$. Since $f\in \mathcal{H}^s(\R^d_+)$, its extension to $\R^d$ by zero belongs to $H^s(\R^d)$ and we can also extend $(\mathcal Uf)_{\xi'}$ by zero to $\R$. A short computation shows that
$$
\frac{1}{\sqrt{2\pi}} \int_\R \left( \mathcal U f\right)_{\xi'} (t) e^{-i\omega t} \,dt
= h^{-d/2} |\xi'|^{1/2} \hat f(h^{-1} \xi',h^{-1}|\xi'|\omega) \,,
$$
and thus,
\begin{align*}
\int_{\R^d} |hp|^{2s}|\hat f(p)|^2 \,dp & = \int_{\R^{d-1}} \left( \int_{\R} \left( |hp'|^2 + (hp_d)^2\right)^s |\hat f(p',p_d)|^2 \,dp_d \right) \,dp' \\
& = h^{-d} \int_{\R^{d-1}} |\xi'|^{1+2s} \left( \int_{\R} \left( 1 + \omega^2\right)^s |\hat f(h^{-1}\xi',h^{-1} |\xi'|\omega)|^2 \,d\omega \right) \,d\xi' \\
& = \int_{\R^{d-1}} |\xi'|^{2s} \left( \int_{\R} \left( 1 + \omega^2\right)^s 
\left| \frac{1}{\sqrt{2\pi}} \int_\R \left( \mathcal U f\right)_{\xi'} (t) e^{-i\omega t} \,dt \right|^2 \,d\omega \right) \,d\xi' \,.
\end{align*}
Since $\left( \mathcal U f\right)_{\xi'}$ vanishes on $(-\infty,0)$, the previous formula can be rewritten as
$$
\int_{\R^d} |hp|^{2s}|\hat f(p)|^2 \,dp = \int_{\R^{d-1}} |\xi'|^{2s} \left\| \left(A^+\right)^{1/2} \left( \mathcal U f\right)_{\xi'} \right\|^2 \,d\xi' \,.
$$
This is equivalent to $\mathcal U \left( H^++1\right) \mathcal U^* = \int_{\R^{d-1}}^\oplus |\xi'|^{2s} A^+ \,d\xi'$ and concludes the proof.
\end{proof}


\subsection{Proof of Proposition \ref{lem:mod}}

Our next step is to state upper and lower bounds on $-\tr \lk \phi H^+ \phi \rk_-$ in terms of the one-dimensional model operators $A$ and $A^+$, in particular, in terms of the function $K(t)$ given in \eqref{eq:K}. As explained below, the main result of this section, Proposition \ref{lem:mod}, will be a direct consequence of the following estimates.

\begin{proposition}\label{lem:modred}
Assume that $\phi \in C_0^1(\R^d)$ is supported in a ball of radius $l =1$ and assume that (\ref{eq:int:gradphi}) is satisfied with $l=1$. Then for any $0 < \delta_2 < \min\{1,2s\}$ there is a constant $C_{\delta_2}$ such that for all $h > 0$ we have
\begin{align}
\label{eq:mod:main1} 
-\tr \lk \phi H^+ \phi \rk_- &\geq  -\Le \int_{\R^d_+} \phi^2(x)  dx h^{-d}  +
\int_{\R^d_+} \phi^2(x)   \frac 1h K\lk \frac{x_d}{h} \rk  dx  h^{-d+1} \, , \\
\label{eq:mod:main2}
-\tr \lk \phi H^+ \phi \rk_- &\leq  -\Le \int_{\R^d_+} \phi^2(x)  dx h^{-d}  +
\int_{\R^d_+} \phi^2(x)   \frac 1h K\lk \frac{x_d}{h} \rk  dx  h^{-d+1}   + C_{\delta_2} h^{-d+1+\delta_2} \,.
\end{align}
\end{proposition}

Assuming Proposition \ref{lem:modred}, we now give the short

\begin{proof}[Proof of Proposition \ref{lem:mod}]
To prove the proposition we may rescale $\phi$ and hence assume $l = 1$. Proposition
\ref{lem:mod} is then an immediate consequence of Proposition \ref{lem:modred}
provided we can show that for any $0<\delta_1<1$ there is a $C_{\delta_1}$ such that for all $h>0$
\begin{equation}
 \label{eq:modk}
\left| \int_{\R_+^{d}} \phi^2(x)   \frac 1h K\lk \frac{x_d}{h} \rk  dx  - \Lz
\int_{\R^{d-1}} \phi^2(x',0) dx' \right| \, \leq \, C_{\delta_1} h^{\delta_1} \,.
\end{equation}

In order to obtain the latter bound, we substitute $x_d = th$ and write,
recalling \eqref{eq:mod:lz},
$$
\int_{\R_+^{d}} \phi^2(x)   \frac 1h K\lk \frac{x_d}{h} \rk  dx - \Lz 
\int_{\R^{d-1}} \phi^2(x',0) dx' = \int_0^\infty K(t)  \int_{\R^{d-1}} 
\int_0^{th} \partial_\tau \phi^2(x',\tau)  d\tau dx' dt \, .
$$
By H\"older's inequality we can further estimate
\begin{align*}
\left| \int_{\R^{d-1}}  \int_0^{th} \partial_\tau \phi^2(x',\tau)  d\tau dx'
\right|
& \leq \left( \int_0^{th} d\tau \right)^{\delta_1}
\left( \int_0^\infty \left| \int_{\R^{d-1}}   \partial_\tau \phi^2(x',\tau) dx'
\right|^{(1-\delta_1)^{-1}} d\tau \right)^{1-\delta_1} \\
& \leq C t^{\delta_1} h^{\delta_1} \,.
\end{align*}
Since $\int_0^\infty t^{\delta_1} |K(t)|\,dt<\infty$ by Lemma \ref{aest}, we obtain inequality \eqref{eq:modk}.
\end{proof}

In the following two subsections we shall prove the lower and the upper bound
in Proposition \ref{lem:modred}, respectively.


\subsection{Lower bound on $-\tr \lk \phi H^+ \phi \rk_-$}
\label{ssec:bdlo}

To prove (\ref{eq:mod:main1}) we use the fact that
$$
-\tr \lk \phi H^+ \phi \rk_- \, \geq \, -\tr \lk \phi (H^+)_- \phi \rk \, .
$$
The lower bound follows from this by integrating the identity
\begin{equation}
 \label{eq:bdlocomp}
(H^+)_-(x,x) = h^{-d} \Le - h^{-d} K \lk \frac{x_d}h \rk \,,
\end{equation}
against $\phi^2$. Equation \eqref{eq:bdlocomp} is a consequence of \eqref{eq:densh+}. Indeed,
by the same argument as in Subsection \ref{ssec:hshl} we learn that
$$
(H_0)_-(x,x) \, = \, \frac{1}{(2\pi)^{d-1}} \frac{1}{h^d} \int_{\R^{d-1}}
|\xi'|^{1+2s} \,  a \lk |\xi'|^{-2s} \rk
d\xi' \, .
$$
On the other hand, by direct diagonalization as in Subsection
\ref{sec:bu:lo} we find that
$$
(H_0)_-(x,x) = h^{-d} \Le \,.
$$
Comparing these two identities with \eqref{eq:densh+} we arrive at
\eqref{eq:bdlocomp}, thus establishing \eqref{eq:mod:main1}.


\subsection{Upper bound on $-\tr \lk \phi H^+ \phi \rk_-$}
\label{ssec:bdup}

To prove  (\ref{eq:mod:main2}) we set $\gamma = (H^+)_-^0$. Its integral kernel is given by \eqref{eq:projh+} in
terms of the kernel $e^+(\cdot,\cdot,\mu)$ of $(A^+-\mu)_-^0$. By the variational principle it follows that 
\begin{align}
\nonumber
- \tr \lk \phi  H^+  \phi \rk_- \, \leq \, & \tr \lk \phi \gamma \phi H^+ \rk \\
\nonumber
\, = \, & \frac{1}{h^{2d}}  \int_{\R^d_+} 
\int_{\R^d_+} \int_{\R^{d-1}} \int_{\R^d} |\xi'| e^{i \xi'\cdot(x'-y') /h}
\, e^+ \lk x_d |\xi'| h^{-1} , y_d |\xi'| h^{-1} , |\xi'|^{-2s} \rk \\
\label{eq:bdup:basic}
& \times \lk |p|^{2s}-1 \rk e^{ip\cdot (y-x)/h} \, \phi(x) \, \phi(y) \,
\frac{dp \, d\xi' \, dx \, dy}{(2\pi)^{2d-1}} \, .
\end{align}
We insert the identity
$$
\phi(x) \, \phi(y) \, = \, \frac 12 \lk \phi^2(x) + \phi^2(y) - |\phi(x) - \phi(y)|^2 \rk \, ,
$$
and by a similar argument as in the proof of Lemma \ref{lem:bu:up} we can use the symmetry in $x$ and $y$ and substitute $q = p_d/|p'|$ to obtain
$$
- \tr \lk \phi H^+ \phi \rk \, \leq \, I_h[\phi] - R_h[\phi]
$$
with the main term
\begin{align*}
 I_h[\phi] = & \frac{1}{h^{2d}}  \int_{\R^d_+}  \int_{\R^d_+} \int_{\R^{d-1}}
\int_{\R^{d-1}}
\int_\R |\xi'| e^{i (\xi'-p')\cdot(x'-y') /h}  e^+ \lk x_d
|\xi'| h^{-1} , y_d |\xi'| h^{-1}, |\xi'|^{-2s} \rk \\ 
& \times  e^{i(y_d-x_d)|p'|q/h} \lk (q^2+1)^s-|p'|^{-2s} \rk |p'|^{1+2s}
\phi^2(x) \, \frac{dq \, dp' \, d\xi' \, dx \, dy}{(2\pi)^{2d-1}}
\end{align*}
and the remainder
\begin{align*}
R_h[\phi]  = & \frac{1}{h^{2d}}  \int_{\R^d_+}  \int_{\R^d_+} \int_{\R^{d-1}}
\int_{\R^d}  |\xi'| e^{i \xi'\cdot(x'-y') /h} \, e^+ \lk x_d
|\xi'| h^{-1} , y_d |\xi'| h^{-1}, |\xi'|^{-2s} \rk \\
& \times|p|^{2s}  e^{ip\cdot(y-x)/h} \, |\phi(x) - \phi(y)|^2 \, \frac{dp \,
d\xi' \, dx \, dy}{2(2\pi)^{2d-1}} \, .
\end{align*}
Since $\phi \in C_0^1(\R^d)$ we can perform the $y'$-integration in
$I_h[\phi]$. We use the fact that
$$
\int_\R \int_0^\infty e^+ \lk x_d,y_d,\mu \rk \lk (q^2+1)^s - \mu \rk e^{i(y_d-z_d)q} \,dy_d\,dq = -2 \pi a^+(x_d,z_d,\mu)
$$
and obtain
\begin{align*}
I_h[\phi]  =  & \frac{1}{h^{d+1}}
\int_{\R^d_+} \int_0^\infty \int_{\R^{d-1}} \int_\R |\xi'|^{2s+2}  \,
e^+ \lk x_d |\xi'| h^{-1},y_d |\xi'| h^{-1},|\xi'|^{-2s} \rk \\
& \times \lk (q^2+1)^s - |\xi'|^{-2s} \rk e^{i(y_d-x_d)|\xi'|q/h} \phi^2(x) \,
\frac{dq \, d\xi' \, dy_d \, dx}{(2\pi)^d} \\
=  & - \frac{1}{h^d} \int_{\R^d_+} \phi^2(x) \int_{\R^{d-1}} |\xi'|^{2s+1} \,
a^+ \lk x_d |\xi'| h^{-1}, |\xi'|^{-2s} \rk \frac{d\xi' \, dx}{(2\pi)^{d-1}} \,
.
\end{align*}
Using again \eqref{eq:bdlocomp} we find that
\begin{equation}
 \label{eq:ih}
I_h[\phi] = - \Le \int_{\R^d_+} \phi^2(x) \, dx \, h^{-d}  + \int_{\R^d_+}
\phi^2(x)  \, K\lk \frac{x_d}{h} \rk dx \, h^{-d} \,.
\end{equation}

It remains to study $R_h[\phi]$. We claim that for any $\frac 12 -s <
\sigma < \min\{\frac 12, 1-s\}$ there is a $C_\sigma$ such that
\begin{equation}
\label{result}
| R_h [\phi] | \, \leq \, C_\sigma h^{-d+2s+2\sigma}
\end{equation}
for all $h>0$. This, together with \eqref{eq:ih} will complete the proof of
\eqref{eq:mod:main2}.

In order to show \eqref{result} we perform the $p$ integration and find
that
\begin{align*}
R_h[\phi] \,  =  \, -\frac{C}{h^{d-2s}} \int_{\R^d_+} \int_{\R^d_+}
\int_{\R^{d-1}} & |\xi'| e^{i\xi'\cdot(x'-y')/h} e^+\lk
\frac{x_d|\xi'|}{h},\frac{y_d|\xi'|}{h},\frac{1}{|\xi'|^{2s}}\rk  \\
& \times \frac{|\phi(x) - \phi(y)|^2}{|x-y|^{d+2s}} d\xi' \, dx \, dy \, .
\end{align*}
We insert
$$
e^{i\xi' \cdot (x'-y')/h} \, = \, \frac{h^{2\sigma}}{|\xi'|^{2\sigma}} (
-\Delta_{x'})^\sigma e^{i\xi'\cdot (x'-y')/h} 
$$
and integrate by parts to get
\begin{align*}
R_h[\phi]  \, = \, -\frac{C}{h^{d-2s-2\sigma}} \int_{\R^d_+} \int_{\R^d_+}
\int_{\R^{d-1}} & |\xi'|^{1-2\sigma} e^{i\xi'\cdot(x'-y')/h} 
e^+\lk \frac{x_d|\xi'|}{h},\frac{y_d|\xi'|}{h},\frac{1}{|\xi'|^{2s}}\rk d\xi' 
\\
& \times  ( -\Delta_{x'})^\sigma  \frac{|\phi(x) - \phi(y)|^2}{|x-y|^{d+2s}} dx \,
dy \, .
\end{align*}
By Lemma \ref{eest} and the fact that $e^+(t,u,\mu)=0$ for $\mu\leq 1$ we
arrive at
\begin{align*}
 |R_h[\phi]| \leq & \frac{C}{h^{d-2s-2\sigma}} \int_{\R^d_+} \int_{\R^d_+}
\int_{\{\xi'\in\R^{d-1}:|\xi'|<1\}} |\xi'|^{-2\sigma} d\xi' \left| (
-\Delta_{x'})^\sigma \frac{|\phi(x) - \phi(y)|^2}{|x-y|^{d+2s}} \right| dx 
dy \\
\leq & \frac{C}{h^{d-2s-2\sigma}} \int_{\R^d_+} \int_{\R^d_+}
\left| ( -\Delta{_x'})^\sigma  \frac{|\phi(x) - \phi(y)|^2}{|x-y|^{d+2s}} \right|
dx dy \, .
\end{align*}
According to Lemma \ref{lem:remest} this implies \eqref{result} and hence
completes the proof of \eqref{eq:mod:main2}.


\section{Local asymptotics near the boundary}
\label{sec:boundary}

In this section we prove Proposition \ref{pro:boundary}. After having analyzed
the half-space case in the previous section, we now show how the case of a
general domain follows. We shall transform the operator $H_\Omega$ locally to an
operator given on the half-space $\R^d_+ = \{(y',y_d)\in \R^{d-1}\times \R \, :
\, y_d > 0 \}$ and we shall quantify the error made by this straightening of the
boundary.

Under the conditions of Proposition \ref{pro:boundary}, let $B$ denote the open
ball of radius $l > 0$, containing the support of $\phi$.
For $x_0 \in B \cap \partial \Omega$ let $\nu_{x_0}$ be the inner normal unit
vector at $x_0$. We choose a Cartesian coordinate system such that $x_0 = 0$ and
$\nu_{x_0} = (0, \dots, 0, 1)$, and we write $x = (x',x_d) \in \R^{d-1} \times
\R$ for $x \in \R^d$.

For sufficiently small $l > 0$ one can introduce new local coordinates near the
boundary. Let $D$ denote the projection of $B$ on the hyperplane given by $x_d
=0$.
Since the boundary of $\Omega$ is compact and $C^{1,\alpha}$ there is a
constant $c > 0$ such that for $0 < l \leq c$ we can find a real function $f
\in C^{1,\alpha}$ given on $D$, satisfying
$$
\partial \Omega \cap B \, = \, \left\{ (x',x_d) \, : \, x' \in D , x_d = f(x')
\right\} \cap B \, .
$$
The choice of coordinates implies $f(0) = 0$ and $ \nabla f (0) = 0$.
Hence, we can estimate
$$
\sup_{x'\in D} |\nabla f(x')| \, = \, \sup_{x'\in D} |\nabla f(x') - \nabla
f(0)| \, \leq \, C_f \, |x'|^\alpha \, \leq \, C_f \, l^\alpha \, .
$$
Since the boundary of $\Omega$ is compact we can choose a constant $C > 0$,
depending only on $\Omega$, in particular independent of $f$, such that the
bound
\begin{equation}
\label{eq:red_fest}
\sup_{x'\in D} |\nabla f(x')| \, \leq \, C   l^\alpha 
\end{equation}
holds.

We introduce new local coordinates via the diffeomorphism $\varphi \, :
\, D \times \R \to \R^d$, given by
$$
y_j \, = \, \varphi_j(x) \, = \, x_j \quad \textnormal{for} \quad j = 1, \dots,
d-1
$$
and
$$
y_d \, = \, \varphi_d(x) \, = \, x_d - f(x') \, .
$$
Note that the determinant of the Jacobian matrix of $\varphi$ equals $1$ and
that the inverse of $\varphi$ is given on $\textnormal{ran} \, \varphi = D
\times \R$. In particular, we get
$$
\varphi \lk \partial \Omega \cap B \rk \, \subset \, \partial \R^d_+ \, = \, \{
y \in \R^d \, : \, y_d = 0 \} \, .
$$ 

Fix $v \in \mathcal{H}^s(\Omega)$ with support in $\overline{B}$. For $y \in
\textnormal{ran} \, \varphi$ put $\tilde v (y) = v \circ \varphi^{-1}(y)$ and
extend $\tilde v$ by zero to $\R^d$.

\begin{lemma}
\label{lem:str1}
The function $\tilde v$ belongs to $\mathcal{H}^s(\R^d_+)$ and there exist positive constants $c$ and $C$ depneding only on $\Omega$ such that for $0 < l \leq c$ we have
$$
\left| (  \tilde v , (-\Delta)_{\R^d_+}^s \tilde v ) - ( v, (-\Delta)^s_\Omega 
v ) \right| \, \leq  C \, l^\alpha \, \min \left\{ (  \tilde v ,
(-\Delta)_{\R^d_+}^s \tilde v ) , \  ( v, (-\Delta)^s_\Omega  v ) \right\} \, .
$$
\end{lemma}

\begin{proof}
By definition, $\tilde v$ belongs to $H^s(\R^d)$ and for $y \in \R^d \setminus
\overline{\R^d_+}$ we find
$x_d = y_d + f(y') < f(x')$, thus $\tilde v (y) = v(x) = 0$. Therefore $\tilde
v$ belongs to $\mathcal{H}^s(\R^d_+)$.

Using the new local coordinates we get
\begin{equation}
\label{eq:co_formhalf}
(  v , (-\Delta)_\Omega^s v )  =  C_{s,d} \int_{\R^d} \int_{\R^d} \frac{|v(x) -
v(w)|^{2}}{|x-w|^{d+2s}}  dx \, dw  =  C_{s,d} \int_{\R^d} \int_{\R^d}
\frac{|\tilde v(y) -  \tilde v(z)|^{2}}{|x-w|^{d+2s}}  dy \, dz \, ,
\end{equation}
where $y = \varphi(x)$ and $z = \varphi(w)$, thus $x=(y', y_d + f(y'))$ and $w =
(z', z_d + f(z'))$. Let us write
\begin{align*}
& \left|\frac{1}{|y-z|^{d+2s}}  -  \frac{1}{|x-w|^{d+2s}} \right| \\
& = \, \frac{1}{|y-z|^{d+2s}} \, \left| 1 - \frac{|y-z|^{d+2s}}{\left[ |y'-z'|^2
+ (y_d+f(y')-z_d-f(z'))^2 \right]^{d/2+s}} \right| \, .
\end{align*}
After multiplying out, the last fraction equals
$$
\lk 1 + \frac{(f(y')-f(z'))^2 + 2 (y_d-z_d)(f(y')-f(z'))}{|y-z|^2}
\rk^{-(d/2+s)} 
$$
and we can employ (\ref{eq:red_fest}) to estimate
\begin{align*}
&\left| \frac{(f(y')-f(z'))^2 + 2 (y_d-z_d)(f(y')-f(z'))}{|y-z|^2} \right| \\
& \leq \, \sup |\nabla f|^2 \, \frac{|y'-z'|^2}{|y-z|^2} + 2 \, \sup |\nabla f|
\, \frac{|y'-z'| \, |y_d-z_d|}{|y-z|^2} \ \leq \ C  l^{\alpha} \, .
\end{align*}
Choosing $l$ small enough we can assume $C l^{\alpha} < 1/2$. Then, combining
the foregoing relations, we find
\begin{equation}
\label{eq:co_absdif}
\left| \frac{1}{|x-w|^{d+2s}} - \frac{1}{|y-z|^{d+2s}} \right| \, \leq \,  C 
\frac{l^{\alpha}}{|y-z|^{d+2s}} \, .
\end{equation}
{F}rom (\ref{eq:co_formhalf}) and (\ref{eq:co_absdif}) we conclude
\begin{eqnarray*}
\lefteqn{ \left| ( \tilde v , (-\Delta)_{\R^d_+}^s \tilde v ) - ( v,
(-\Delta)^s_\Omega  v ) \right| } \\
& \leq & C_{s,d} \iint | \tilde v(y)- \tilde v(z)|^2 \left|
\frac{1}{|y-z|^{d+2s}} - \frac{1}{|x-w|^{d+2s}} \right| \, dy \, dz \\
& \leq & C \, l^{\alpha} (  \tilde v, (-\Delta)^s_{\R^d_+} \tilde v ) \, . 
\end{eqnarray*}
This proves the first claim of the Lemma. The second claim follows by
interchanging the roles of $(-\Delta)^s_{\R^d_+}$ and $(-\Delta)^s_\Omega$.
\end{proof}

On the range of $\varphi$ we define $\tilde \phi = \phi \circ \varphi^{-1}$
and extend it by zero to $\R^d$ such that $\tilde \phi \in C_0^1(\R^d)$ and
$\| \nabla \tilde \phi \|_\infty \leq Cl^{-1}$ hold. Using Lemma
\ref{lem:str1} we show the following relations.

\begin{lemma}
\label{lem:str}
For $0 < l \leq c$ and any $h > 0$ the estimate
\begin{equation}
\label{eq:str:lem1}
\left| \tr ( \phi H_\Omega \phi )_- - \tr ( \tilde \phi H^+ \tilde \phi )_-
\right| \, \leq \, C \, l^{d+\alpha} \, h^{-d}
\end{equation}
holds. Moreover, we have 
\begin{equation}
\label{eq:str:lem2}
\int_\Omega \phi^2(x) \, dx  \, = \, \int_{\R^d_+} \tilde \phi^2(y) \, dy
\end{equation}
and 
\begin{equation}
\label{eq:str:lem3}
0 \leq \int_{\partial \Omega} \phi^2(x) \, d\sigma(x) - \int_{\R^{d-1}}
\tilde \phi^2(y',0) \, dy' \, \leq \, C \, l^{d-1+2\alpha} \, .
\end{equation}
\end{lemma}

\begin{proof}
The definition of $\tilde \phi$ and the fact that the Jacobian of $\phi$ equals $1$ immediately gives (\ref{eq:str:lem2}). Using (\ref{eq:red_fest}) we estimate
$$
\int_{\partial\Omega} \phi^2(x) \, d\sigma(x) \, = \, \int_{\R^{d-1}} \tilde
\phi^2(y',0) \sqrt{1+|\nabla f|^2 } \, dy' \, \leq \, \int_{\R^{d-1}} \tilde
\phi^2(y',0) \, dy' +C l^{d-1+2\alpha} \, .
$$
from which (\ref{eq:str:lem3}) follows.

To prove (\ref{eq:str:lem1}) we refer to the variational principle once more and note that
$$
- \tr \lk \phi H_\Omega \phi \rk_- \, = \, \inf_{0 \leq \gamma \leq 1} \tr \lk
\phi \gamma \phi H_\Omega \rk \, ,
$$
where we can assume that infimum is taken over trial density matrices $\gamma$ supported in $\overline{B} \times
\overline{B}$. Fix such a $\gamma$. 
For $y$ and $z$ from  $D \times \R$ set
$$
\tilde \gamma (y,z) \, = \, \gamma \lk \varphi^{-1}(y), \varphi^{-1}(z) \rk \, ,
$$
so that $0 \leq \tilde \gamma \leq  1$ and the range of $\tilde \gamma$ belongs
to the form domain of $\tilde \phi H^+ \tilde \phi$.
According to Lemma \ref{lem:str1} it follows that
\begin{eqnarray*}
\tr \lk \phi \gamma \phi H_\Omega \rk & \geq & \tr \lk \tilde \phi \tilde \gamma
\tilde \phi \lk h^{2s} (1-Cl^\alpha )  (-\Delta)^s_{\R^d_+} - 1 \rk \rk \\
& \geq & - \tr \lk \tilde \phi \lk (1-Cl^\alpha ) h^{2s} (-\Delta)^s_{\R^d_+} -
1 \rk \tilde \phi \rk_- 
\end{eqnarray*}
and consequently
$$
\tr \lk \phi H_\Omega \phi \rk_- \, \leq \, \tr \lk \tilde \phi  \lk
(1-Cl^\alpha ) h^{2s} (-\Delta)^s_{\R^d_+} - 1 \rk \tilde \phi \rk_- \, .
$$

Set $\varepsilon = 2C l^\alpha$ and assume $l$ to be sufficiently small, so that
$0 < \varepsilon \leq 1/2$. Then
\begin{align*}
\tr \lk \phi H_\Omega \phi \rk_- \, &\leq \,  \tr \lk \tilde \phi \lk
(1-Cl^\alpha) h^{2s} (-\Delta)^s_{\R^d_+} -1 \rk \tilde \phi \rk_-  \\
& \leq \, \tr \lk \tilde \phi \lk  (-h^2 \Delta)^s_{\R^d_+} - 1 \rk \tilde \phi
\rk_- + \tr \lk \tilde \phi \lk (\varepsilon - Cl^\alpha ) h^{2s}
(-\Delta)^s_{\R^d_+} - \varepsilon \rk \tilde \phi \rk_- \\
& \leq \, \tr ( \tilde \phi H^+ \tilde \phi )_- + \varepsilon \tr \lk \tilde
\phi \lk (h^{2s}/2) (-\Delta)^s_{\R^d_+} - 1 \rk \tilde \phi \rk_- \, .
\end{align*}
Using Lemma \ref{lem:berezin} we estimate $\tr (\tilde \phi ((h^{2s}/2)
(-\Delta)^s_{\R^d_+} - 1 ) \tilde \phi )_- \leq C l^d h^{-d}$
and  it follows that
$$
\tr ( \phi H_\Omega \phi )_- \, \leq \, \tr ( \tilde \phi H^+ \tilde \phi )_- +
C \, l^{d+\alpha} \, h^{-d} \, .
$$
Finally, by interchanging the roles of $H_\Omega$ and $H^+$, we get an analogous
lower bound and the proof of the Lemma is complete.
\end{proof}

We conclude this section by giving the short

\begin{proof}[Proof of Proposition \ref{pro:boundary}]
 It suffices to combine Lemma \ref{lem:str} and Proposition \ref{lem:mod}.
\end{proof}


\section{Localization}
\label{sec:loc} 

In this section we construct the family of localization functions $(\phi_u)_{u \in \R^d}$ and prove Proposition \ref{pro:loc}. 
Fix a real-valued function $\phi \in C_0^\infty(\R^d)$ with support in the ball $\{x \in \R^d \, : \, |x| < 1\}$ that satisfies $\| \phi \|_2 = 1$. We recall the definition of the local length scale $l(u)$ from \eqref{eq:l}.
For $u, x \in \R^d$ let $J(x,u)$ be the Jacobian of the map $u \mapsto (x-u)/l(u)$. We define 
$$
\phi_u(x) \, = \, \phi \lk \frac{x-u}{l(u)} \rk \sqrt{J(x,u)} \, l(u)^{d/2} \, ,
$$
such that $\phi_u$ is supported in the ball $B_u = \{ x \in \R^d \, : \, |x-u| < l(u) \}$.

By definition, the function $l(u)$ is smooth and satisfies $0 < l(u) \leq 1/2$ and $\left\| \nabla l \right\|_\infty \leq 1/2$.  Therefore, according to \cite{SolSpi03}, the functions $\phi_u$ satisfy (\ref{eq:int:grad}) and (\ref{eq:int:unity}) for all $u \in \R^d$. 

To prove the lower bound in Proposition \ref{pro:loc} we follow some ideas from \cite{LieYau88}. In particular, we need the following auxiliary results; the first one gives an IMS-type localization formula for the fractional Laplacian.

\begin{lemma}
\label{lem:loc_ims}
For the family of functions $\lk \phi_u \rk_{u \in \R^d}$ introduced above and for all $f \in \mathcal{H}^s(\Omega)$ the identity
$$
\lk f, (-\Delta)^s f \rk \, = \, \int_{\Omega^*} \lk \phi_u f , (-\Delta)^s \phi_u f \rk l(u)^{-d} \, du - \lk f, L f \rk 
$$
holds with $\Omega^* = \{ u \in \R^d \, : \, \textnormal{supp} \phi_u \cap \Omega \neq \emptyset \}$.
The operator $L$ is of the form 
\begin{equation}
\label{eq:loc_kernel}
L \, = \, \int_{\Omega^*} L_{\phi_u} \, l(u)^{-d} \, du \, ,
\end{equation}
where $L_{\phi_u}$ is a bounded operator with integral kernel
$$
L_{\phi_u}(x,y) \, = \, C_{s,d} \, \frac{|\phi_u(x) - \phi_u(y)|^2}{|x-y|^{d+2s}} \chi_\Omega(x)\chi_\Omega(y) \, .
$$
\end{lemma}

Here $\chi_\Omega$ denotes the characteristic function of $\Omega$.

Lemma \ref{lem:loc_ims} implies that for any operator $\gamma$ with range in $\mathcal{H}^s(\Omega)$
\begin{equation}
\label{eq:loc_trace}
\textnormal{Tr} \, \gamma (-\Delta)^s \, = \, \int_{\R^d} \tr \lk \gamma \phi_u (-\Delta)^s \phi_u \rk l(u)^{-d } \, du - \tr \, \gamma L \, .
\end{equation}
The next result allows to estimate the localization error $\tr\,\gamma L$.

\begin{lemma}
\label{lem:loc_rem}
For $u \in \R^d$ and $0 < \delta \leq 1/2$ we have
$$
\tr  \, \gamma L_{\phi_u}  \, \leq \, \tr \, \gamma \lk C \delta^{2-2s}
l(u)^{-2s}  \chi_\delta \chi_\Omega \rk + C \, \| \gamma \| \, l(u)^{-2s} \,
\delta^{-d+2-2s}   r(\delta)
$$
with
$$
r(\delta) \, = \, \left\{ \begin{array}{ll} 1 & \textnormal{if} \  1-d/4 < s
<1 \\
|\ln \delta| & \textnormal{if} \ 0< s = 1-d/4  \\
\delta^{d+4s-4} & \textnormal{if} \  0< s <1-d/4
\end{array} \right. \, .
$$
where $\chi_\delta$ denotes the characteristic function of $\{ x \in \R^d
\, : \, |x-u| < l(u) (1+\delta) \}$.
\end{lemma}

\begin{proof}
By translation and scaling we can assume that $u = 0$ and $l(u) = 1$, and hence
$\phi_u = \phi$. (This rescaling changes $\Omega$, but the bound we are
going to prove is independent of the domain and therefore not affected by this
dilation.) We set
$$
L^1_{\phi}(x,y) \, = \, \left\{ \begin{array}{lr} L_{\phi}(x,y) \,
\chi_\delta(x) \, \chi_\delta(y) & \mathrm{if}\ |x-y| < \delta \\
0 & \mathrm{if}\ |x-y| \geq \delta \end{array} \right. \, ,
$$
$L^0_{\phi}(x,y) = L_{\phi}(x,y) - L^1_{\phi}(x,y)$ and $\theta(x) =
\int L_{\phi}^1(x,y) \, dy$. By a simple adaption of the arguments of \cite[Thm. 10]{LieYau88} we find that
for any $\epsilon>0$
\begin{equation}
\label{eq:loc_trlphi}
\tr \, \gamma L_{\phi} \, \leq \, \tr \, \gamma \lk \theta + \varepsilon \, 
\chi_0 \rk + \frac{\| \gamma \|}{2 \varepsilon } \,  \tr \lk L^0_{\phi} \rk^2
\, .
\end{equation}
It remains to bound $\theta$ and $\tr ( L^0_{\phi} )^2$.

We begin by estimating $\theta$.
By definition, for $|x| \geq 1 + \delta$ we have $L^1_{\phi}(x,y) = 0$ and hence
$\theta(x) = 0$, and for $|x| < 1 + \delta$ we get
$$
\theta(x) \!=\! C_{s,d}\! \int_{\begin{subarray}{l} |x-y|< \delta \\ |y| < 1 +
\delta \end{subarray}} \frac{(\phi(x)- \phi(y))^2 }{|x-y|^{d+2s}} \chi_{\Omega}(x) \chi_\Omega(y) dy  \leq  C \left\| \nabla \phi \right\|^2_\infty \chi_\Omega(x) \!\int_{|x-y|<\delta} \! \frac{1}{|x-y|^{d+2s-2}}  dy \, .
$$
Thus, for all $x \in \R^d$
\begin{equation}
\label{eq:loc_theta}
\theta(x) \, \leq \, C \, \delta^{2-2s} \, \chi_\delta(x) \, \chi_\Omega(x) \, .
\end{equation}

Finally, we estimate $\tr ( L^0_{\phi} )^2$. The symmetry of $L_{\phi}^0(x,y)$ implies
$$
\tr \lk L^0_{\phi} \rk^2 \, \leq \,  C \iint_A   \lk \frac{(\phi(x)- \phi(y))^2}{|x-y|^{d+2s}} \rk^2 dx \, dy \, 
$$
where $A$ denotes the set $\{ (x,y) \in \R^d \times \R^d \, : \, |x| < \min(|y|,
1 ) \, , \, |x-y| \geq \delta \}$. Set $A_1 = \{ (x,y) \in A \, : \, |y| \geq 2
\}$ and $A_2 = \{ (x,y) \in A \, : \, |y| < 2 \}$. Then
$$
\tr \lk L^0_{\phi} \rk^2  \leq  C \iint_{A_1}  \lk
\frac{\phi(x)^4}{|x-y|^{2d+4s}} \rk dx \, dy + C \, \left\| \nabla \phi
\right\|^4_\infty \, \iint_{A_2} \frac{1}{|x-y|^{2d+4s-4}} \, dx \, dy \, .
$$
The right-hand side is bounded by $C \delta^{-d-4s+4}$ for $1-d/4 < s <1$, by $C |\ln \delta|$ for $0 < s = 1-d/4$, and by $C$ for $0 < s < 1-d/4$.
Finally, we choose $\varepsilon = \delta^{2-2s}$ and combining the last estimates with (\ref{eq:loc_trlphi}) and (\ref{eq:loc_theta}) yields the claimed result.
\end{proof}

\begin{proof}[Proof of Proposition \ref{pro:loc}]
We apply Lemma \ref{lem:loc_rem} with a parameter $0<\delta_u\leq 1/2$ to be specified
later. For ease of notation we write $\chi_u$ instead of $\chi_{\delta_u}$.
Identities (\ref{eq:loc_kernel}) and (\ref{eq:loc_trace}) and the estimate from
Lemma \ref{lem:loc_rem} imply
\begin{align}
\nonumber
\tr \, \gamma (-\Delta)^{s} \, \geq \, & \int_{\Omega^*} \tr \, \gamma \lk
\phi_u(-\Delta)^{s} \phi_u - C \delta_u^{2-2s} l(u)^{-2s} \chi_u \,
\chi_\Omega \rk l(u)^{-d} \, du \\
\label{eq:loc_trgamma}
&- \, C \, \| \gamma \| \,  \int_{\Omega^*} \delta_u^{-d+2-2s} r(\delta_u)
l(u)^{-d-2s}\, du \, .
\end{align}

If the supports of $\chi_u$ and $\phi_{u'}$ overlap, we have $|u-u'| \leq (3/2)
l(u) + l(u')$. It follows that $l(u') - l(u) \leq \|\nabla l \|_\infty \lk
(3/2) l(u) + l(u') \rk$. Since $\| \nabla l \|_\infty \leq 1/2$ we find $l(u')
\leq C l(u)$ and $l(u)^{-1} \leq C l(u')^{-1}$. Similarly, we get $l(u) \leq C
l(u')$. We assume now that $\delta_u$ satisfies
\begin{equation}
 \label{eq:deltachoice}
\delta_u \leq C \delta_{u'}
\quad\mathrm{if}\ |u-u'| \leq (3/2) (l(u) + l(u')) \,.
\end{equation}
Using these locally uniform bounds on $l(u)/l(u')$ and $\delta_u/\delta_{u'}$,
together with (\ref{eq:int:unity}), we can deduce the pointwise bound for all $x\in\R^d$
\begin{align*}
\int_{\Omega^*}\delta_u^{2-2s} l(u)^{-2s} \, \chi_u(x) \, \chi_\Omega(x) \, 
\frac{du}{l(u)^d} 
& = \int_{\Omega^*} \delta_u^{2-2s} l(u)^{-2s} \, \chi_u (x) \, \chi_\Omega(x)
\lk \int \phi_{u'}(x)^2 \frac{du'}{l(u')^d} \rk \frac{du}{l(u)^d} \\
& \leq C \int_{\Omega^*} \phi_{u'}(x) \,\delta_{u'}^{2-2s}  l(u')^{-2s} \,
\phi_{u'}(x) \, \frac{du'}{l(u')^d} \,.
\end{align*}
Rewriting the last integral with $u$ as integration variable, in view of (\ref{eq:loc_trgamma}), we find
$$
\tr \, \gamma (-\Delta)^{s} \geq \int_{\Omega^*} \!\textnormal{Tr} \, \gamma
\lk \phi_u \lk (-\Delta)^{s} - \frac{C \delta_u^{2-2s}}{l(u)^{2s}} \rk \phi_u
\rk \frac{du}{l(u)^d} - C \|\gamma\| \int_{\Omega^*} \delta_u^{-d+2-2s}
r(\delta_u) \frac{du}{l(u)^{d+2s}} \, .
$$
By the variational principle it follows that
\begin{eqnarray}
\nonumber
\tr (H_\Omega)_-   &=& -  \inf_{0 \leq \gamma \leq 1} \tr \, \gamma \lk (-h^2 \Delta)^{s}- 1 \rk \\
\nonumber
& \leq & \int_{\Omega^* } \tr \lk \phi_u \lk (-h^2 \Delta)^{s} -1 - C  h^{2s}
\delta_u^{2-2s}l(u)^{-2s} \rk \phi_u \rk_- \frac{du}{l(u)^d} \\ 
\label{eq:loc_trh}
&& + C h^{2s}  \int_{\Omega^*} \delta_u^{-d+2-2s}
r(\delta_u)  \frac{du}{l(u)^{d+2s}}  \, . 
\end{eqnarray}
To bound the first term, we use Lemma \ref{lem:berezin}. For any $u \in \R^d$,
let $\rho_u$ be another parameter satisfying $0< \rho_u \leq 1/2$ and estimate
\begin{eqnarray*}
\lefteqn{ \textnormal{Tr} \lk \phi_u \lk (-h^2\Delta)^{s} -1 - C h^{2s}
\delta_u^{2-2s} \,  l(u)^{-2s} \rk \phi_u \rk_- } \\
& \leq & \textnormal{Tr} \lk \phi_u H_\Omega \phi_u \rk_- + C \, \textnormal{Tr}
\lk \phi_u \lk \rho_u  h^{2s} (-\Delta)^{s} - \rho_u- h^{2s} \delta_u^{2-2s}
l(u)^{-2s}  \rk \phi_u \rk_- \\
&\leq& \textnormal{Tr} \lk \phi_u  H_\Omega \phi_u \rk_- + C \, l(u)^d (\rho_u
h^{2s})^{-d/(2s)} \lk \rho_u + h^{2s} \delta_u^{2-2s} l(u)^{-2s} \rk^{1+d/(2s)}
\, .
\end{eqnarray*}
We pick $\rho_u=h^{2s} \, \delta_u^{2-2s} \, l(u)^{-2s}$. By \eqref{eq:int:lulo}
and our assumption that $\delta_u\leq 1/2$, we see that $\rho_u \leq
(h/l_0)^{2s} 2^{6s-2}$. We assume now that $h\leq C^{-1} l_0$ (with a possibly
large constant $C$) in order to guarantee that $\rho_u\leq 1/2$. With this
choice we find
\begin{equation}
\label{eq:loc_trfirstterm}
\textnormal{Tr} \lk \phi_u \lk (-h^2\Delta)^{s} -1 - \frac{C h^{2s}
\delta_u^{2-2s}}{l(u)^{2s}} \rk \phi_u \rk_-  \, \leq \, \textnormal{Tr} \lk
\phi_u  H_\Omega \phi_u \rk_- + C \,  \frac{\delta_u^{2-2s} \,
l(u)^{d-2s}}{h^{d-2s}} \, .
\end{equation}
Combining (\ref{eq:loc_trh}) and (\ref{eq:loc_trfirstterm}) we obtain
\begin{equation}
\label{eq:loc_trfinal}
\textnormal{Tr}  (H_\Omega)_-  \leq   \int_{\Omega^*} \textnormal{Tr} \lk \phi_u H_\Omega \phi_u \rk_- \frac{du}{l(u)^d} 
+  C \int_{\Omega^*} \lk  \frac{\delta_u^{2-2s}}{h^{d-2s}l(u)^{2s}} +
\frac{h^{2s}\delta_u^{-d+2-2s} r(\delta_u)}{l(u)^{d+2s}}  \rk du \, .
\end{equation}

At this point we choose $\delta_u$ in order to minimize the second integrand, which
we shall denote by $I_u$. We pick
$$
\delta_u \, = \, \left\{ \begin{array}{ll}
h/l(u) & \textnormal{if} \  1-d/4 < s <1 \\
(h/l(u)) |\ln(l(u)/h)|^{1/(4-4s)} & \textnormal{if} \ 0< s = 1-d/4  \\
(h/l(u))^{d/(4-4s)} & \textnormal{if} \  0< s <1-d/4
\end{array} \right.
$$
and note that $\delta_u \leq 1/2$ if $h \leq C^{-1} l_0$ by \eqref{eq:int:lulo}.
Moreover, \eqref{eq:deltachoice} is an easy consequence of the corresponding
estimate for $l(u)/l(u')$. With this choice we arrive at the bounds
$$
I_u \, \leq \, C \left\{ \begin{array}{ll}
h^{-d+2} l(u)^{-2} & \textnormal{if} \  1-d/4 < s <1 \\
h^{-d+2} l(u)^{-2} |\ln (l(u)/h)|^{1/2} & \textnormal{if} \ 0< s = 1-d/4  \\
h^{-d/2+2s} l(u)^{-d/2-2s} & \textnormal{if} \  0< s <1-d/4
\end{array} \right. \, .
$$
Finally, we integrate with respect to $u$. The same arguments that lead to
(\ref{eq:int:U1}) and (\ref{eq:int:U2}) yield
$$
\int_{\Omega^*} I_u\, du \, \leq \, C \left\{ \begin{array}{ll}
h^{-d+2} l_0^{-1} & \textnormal{if} \  1-d/4 < s <1 \\
h^{-d+2} l_0^{-1} |\ln (l_0/h) |^{1/2} & \textnormal{if} \ 0< s = 1-d/4  \\
h^{-d/2+2s} l_0^{-d/2-2s+1} & \textnormal{if} \  0< s <1-d/4
\end{array} \right. \, .
$$
This completes the proof of the lower bound with the remainder stated in
Proposition \ref{pro:loc}.

To prove the upper bound we put
$$
\gamma \, = \, \int_{\R^d} \phi_u \, \lk \phi_u H_\Omega \phi_u \rk_-^0 \, \phi_u \, l(u)^{-d} \, du \, .
$$
Obviously, $\gamma \geq 0$ holds and in view of (\ref{eq:int:unity}) also $\gamma \leq 1$. The range of $\gamma$ belongs to $\mathcal{H}^s(\Omega)$ and by the variational principle it follows that 
$$
- \tr (H_\Omega)_- \, \leq \, \tr \, \gamma H_\Omega \, =  \, -  \int_{\R^d} \tr \lk \phi_u H_\Omega \phi_u \rk_- l(u)^{-d} \, du \, .
$$ 
This yields the upper bound and finishes the proof of Proposition \ref{pro:loc}.
\end{proof}


\section{Discussion of the second term}
\label{sec:snd}

\subsection{Representations for the second constant}
In this section we study the second term of \eqref{eq:main} in more detail. First we derive representation \eqref{eq:main2}.

\begin{proposition}\label{const}
One has
\begin{align}\label{eq:const}
\Lz & = \int_{\R^{d-1}} \zeta(|p'|^{-2s})\, \frac{dp'}{(2\pi)^{d-1}} \notag \\
& = \frac{|\Sph^{d-2}|}{(2\pi)^{d-1}} \frac{2s}{(d-1)(d-1+2s)} \, \tr \left[ \chi A^{-(d-1)/2s} \chi - (A^+)^{-(d-1)/2s} \right] \,. 
\end{align}
Here $\chi$ is the characteristic function of $\R_+$ and 
\begin{equation}
\label{eq:zeta}
\zeta(\mu) = \mu^{-1} \int_0^\infty \lk  a(\mu)-a^+(t,\mu) \rk \,dt \,.
\end{equation}
\end{proposition}

\begin{proof}
The first identity follows immediately from \eqref{eq:K} and \eqref{eq:mod:lz}. The second identity follows from the fact that
$$
\int_{\R^{d-1}} |p'|^{2s} (E-|p'|^{-2s})_- \frac{dp'}{(2\pi)^{d-1}} 
= \frac{|\Sph^{d-2}|}{(2\pi)^{d-1}} \frac{2s}{(d-1)(d-1+2s)} E^{-(d-1)/2s}
$$
for any $E>0$, which by the spectral theorem implies that
$$
\int_{\R^{d-1}} |p'|^{2s} a^+(t,|p'|^{-2s}) \frac{dp'}{(2\pi)^{d-1}} 
= \frac{|\Sph^{d-2}|}{(2\pi)^{d-1}} \frac{2s}{(d-1)(d-1+2s)} (A^+)^{-(d-1)/2s}(t,t)
$$
and similarly for $A$.
\end{proof}

\begin{remark}
There is another representation, namely,
\begin{equation}
\label{eq:constxi}
\Lz = \frac{2s}{d-1+2s} \int_{\R^{d-1}} \xi(|p'|^{-2s})\,
\frac{dp'}{(2\pi)^{d-1}} \,,
\end{equation}
where
\begin{equation}\label{eq:xi}
\xi(\mu) = \int_0^\infty \lk  e(\mu)-e^+(t,\mu) \rk \,dt \,.
\end{equation}
Here $e(\mu)$ and $e^+(t,\mu)$ are the diagonals of the integral kernels of the spectral projcections $(A-\mu)_-^0$ and $(A^+-\mu)_-^0$, respectively. We have not shown that the integral in \eqref{eq:xi} converges, since we will not use \eqref{eq:constxi} in the remainder of this paper. Identity \eqref{eq:constxi} is an easy consequence of \eqref{eq:const} and the fact that
$$
a(\mu) = \int_0^\mu e(\tau) \,d\tau \,
\qquad
a^+(t,\mu) = \int_0^\mu e^+(t,\tau) \,d\tau
$$
which follows by the spectral theorem from $(E-\mu)_- = \int_0^\mu (E-\tau)_-^0 \,d\tau$. Representation \eqref{eq:constxi} is natural since in terms of this function the conjectured formula for the \emph{number} of negative eigenvalues of $H_\Omega$ takes the form
$$
\iint_{T^*\Omega} \lk |p|^{2s}-1 \rk_-^0 \frac{dp dx}{(2\pi h)^d}  -
\iint_{T^*\partial \Omega} \xi(|p'|^{-2s}) \frac{dp' d\sigma(x)}{(2\pi h)^{d-1}}
+ o(h^{-d+1}) \, ,
$$
which is the analogue of well-known two-term semi-classical formulas in the
local case; see, for instance, \cite{Ivrii80a,SaVa}. The function $\xi$ plays
the role of a spectral shift. Note that we avoided to write \eqref{eq:zeta} and
\eqref{eq:xi} in terms of a trace. While the integrals on the diagonals
converge, we do not expect the operators to be trace class, see \cite{Pu}.
\end{remark}

\begin{remark}\label{rephp}
Yet another representation is
$$
h^{-d+1} \Lz = \int_0^\infty \left( H_-(x,x) - (H^+)_-(x,x) \right) \,dx_d \,.
$$
(Note that the right side is independent of $x'$.) This follows from \eqref{eq:densh+} and the corresponding formula for $H$. Using this representation one sees that our asymptotic formula coincides with the one obtained in \cite{BanKul08,BaKuSi09}.
\end{remark}

Finally, we refer to \eqref{eq:ap:lz} in the appendix for a representation of $\Lz$ in terms of generalized eigenfunctions of $A^+$.


\subsection{A lemma about operator monotone functions}

To prove that the constant $\Lz$ is positive we shall make use of the following

\begin{lemma}\label{song}
Let $B$ be a non-negative operator with $\ker B=\{0\}$ and let $P$ be an orthogonal projection. Then for any operator monotone function $\phi:(0,\infty)\to \R$,
\begin{equation}
\label{eq:song}
P\phi(PBP)P \geq P\phi(B)P \,.
\end{equation}
If, in addition, $B$ is positive definite and $\phi$ is not affine linear, then $\phi(PBP)=P\phi(B)P$ implies that the range of $P$ is a reducing subspace of $B$.
\end{lemma}

We recall that, by definition, the range of $P$ is a reducing subspace of a non-negative (possibly unbounded) operator if $(B+\tau)^{-1}\ran P\subset\ran P$ for some $\tau>0$. We note that this is equivalent to $(B+\tau)^{-1}$ commuting with $P$, and we see that the definition is independent of $\tau$ since
\begin{align*}
& (B+\tau')^{-1} P - P (B+\tau')^{-1} \\
& \quad = (B+\tau) (B+\tau')^{-1} \left( (B+\tau)^{-1} P - P (B+\tau)^{-1} \right) (B+\tau) (B+\tau')^{-1} \,.
\end{align*}

Hansen \cite{Ha} has proved Lemma \ref{song} for bounded $B$ and without the equality statement. It is not clear how to extend his proof to our general case and we provide a different argument.

For our proof we recall L\"owner's theorem \cite{Do} which characterizes operator monotone functions on $(0,\infty)$ by the representation
\begin{equation}
\label{eq:bernstein}
\phi(E) =a + bE + \int_{[0,\infty)} \frac{\tau E -1}{E+\tau} \, d\rho(\tau)
\end{equation}
with $a\in\R$, $b\geq 0$ and a finite, positive measure $\rho$ on $[0,\infty)$. Note that the function $\phi(E)=E^s$, $0<s< 1$, to which we apply this lemma in the next sections, is operator monotone in view of the representation
$$
E^s = \frac{\sin(\pi s)}{\pi} \int_0^\infty \tau^{s-1} \frac{E}{E+\tau} \,d\tau \,,\qquad 0<s<1 \,.
$$
This is of the form \eqref{eq:bernstein} with $d\rho(\tau)=\!(\sin(\pi s)/\pi)(1+\tau^2)^{-1}\tau^sd\tau$, $a=(\sin(\pi s)/\pi)\int_0^\infty\!\tau^{-1}d\rho(\tau)$ and $b=0$.

\begin{proof}
We first prove that
\begin{equation}
\label{eq:projectors}
P B^{-1} P \, \geq \, P(PBP)^{-1}P \,.
\end{equation} 
Here on the right side, the operator $PBP$ is inverted as an operator on the range of $P$.

By a monotone convergence argument we may assume that $B$ is positive definite. Let $f$ be an arbitrary element in the Hilbert space. For any $\psi$ in the form domain of $B$ we can write
$$
(f,PB^{-1}Pf) = -(\psi,B\psi) + 2\re (Pf,\psi) + \|B^{1/2}\psi - B^{-1/2}Pf\|^2 \,.
$$
We apply this to $\psi=P (PBP)^{-1} Pf$. Note that $\psi$ belongs to the operator domain of $PBP$ and hence also to the form domain of $PBP$, which means that $P\psi=\psi$ belongs to the form domain of $B$. We find
$$
(f,PB^{-1}Pf) = (f,P(PBP)^{-1}Pf) + \|B^{1/2}P(PBP)^{-1}Pf - B^{-1/2}Pf\|^2 \,.
$$
This proves \eqref{eq:projectors}.

Moreover, if equality holds in \eqref{eq:projectors} (still assuming that $B$ is positive definite), then $B^{1/2}P(PBP)^{-1}Pf = B^{-1/2}Pf$ for all $f$, that is, $P(PBP)^{-1}Pf = B^{-1}Pf$ for all $f$, which means that $B^{-1}\ran P\subset\ran P$. Thus, $\ran P$ reduces $B$.

Now assume that $\phi$ is of the form \eqref{eq:bernstein} and rewrite $(\tau E-1)/(E+\tau) = \tau -(\tau^2+1)/(E+\tau)$. By the spectral theorem,
$$
P\phi(B)P = aP + bPBP + \int_{[0,\infty)} \left( \tau P - (\tau^2+1) P(B+\tau)^{-1}P \right) \,d\rho(\tau) \,.
$$
Similarly, $PBP$ is a self-adjoint operator in the range of $P$ and by the spectral theorem in that space
$$
P\phi(PBP)P = aP + bPBP + \int_{[0,\infty)} \left( \tau P - (\tau^2+1) P(PBP+\tau P)^{-1}P \right) \,d\rho(\tau) \,.
$$
Here, as before $PBP+\tau P$ is inverted in the range of $P$. Thus,
\begin{align*}
P\phi(PBP)P - P\phi(B)P 
& = - \int_{[0,\infty)} \lk P(PBP + \tau P)^{-1}P - P(B+\tau)^{-1}P \rk (\tau^2+1) \,d\rho(\tau) \, .
\end{align*}
By \eqref{eq:projectors} with $B$ replaced by $B+\tau$, the integrand is a non-positive operator for every $\tau\in[0,\infty)$. Thus, $\phi(PBP)\geq P\phi(B)P$, as claimed.

This argument show that $\phi(PBP)= P\phi(B)P$ implies $P(PBP + \tau P)^{-1}P = P(B+\tau)^{-1}P$ for $\rho$-a.e. $\tau\in[0,\infty)$. If $\phi$ is not affine linear, then the measure $\rho$ is not identically zero and there is a $\tau\in[0,\infty)$ with $P(PBP + \tau P)^{-1}P = P(B+\tau)^{-1}P$. Now the analysis of equality in \eqref{eq:projectors} (note that $B+\tau$ is positive definite) implies that $\ran P$ reduces $B$.
\end{proof}


\subsection{Positivity of the constant}

Here we shall prove

\begin{proposition}\label{constpos}
For any $0<s<1$ and $d\geq 2$, one has $\Lz>0$.
\end{proposition}

\begin{proof}
We shall show that for arbitrary non-negative operators $B$ with $\ker B=\{0\}$ and orthogonal projections $P$,
\begin{equation}
\label{eq:traceineq1}
\tr \left[ P B^{-\alpha} P - (PBP)^{-\alpha} \right] \geq 0
\qquad\text{for all}\ \alpha>0 \,.
\end{equation}
If $B$ is positive definite, then equality holds iff the range of $P$ is a reducing subspace of $B$.

We apply this to the second representation in \eqref{eq:const} with $B=A$ and $P=\chi$ and note that $A^+=\chi A\chi$. Thus \eqref{eq:traceineq1} implies $\Lz\geq 0$. Since $B\geq 1$ and since the range of $P$ is not a reducing subspace for $B$ (indeed, $(A+\tau)^{-1}f$ does not necessarily vanish on $(-\infty,0)$ if $f$ does), we even have $\Lz>0$, as claimed.

It remains to prove \eqref{eq:traceineq1}. The argument is somewhat different depending on whether $\alpha\leq 1$ or not. In the first case we learn from Lemma \ref{song} with $\phi(E)=-E^{-\alpha}$ that
$$
PB^{-\alpha}P \geq (PBP)^{-\alpha}
$$
with equality iff $\ran P$ reduces $B$. This immediately implies \eqref{eq:traceineq1} and the equality statement. Now assume that $\alpha>1$. Then Lemma \ref{song} with $\phi(E)=-E^{-1/\alpha}$ yields
$$
PBP \geq (PB^{-\alpha} P)^{-1/\alpha}
$$
with equality iff $\ran P$ reduces $B^{-\alpha}$. Since $E\mapsto E^{-\alpha}$ is strictly monotone decreasing, we obtain again \eqref{eq:traceineq1} and, using the spectral theorem, the equality statement.
\end{proof}


\subsection{Comparison with a fractional power of the Dirichlet Laplacian}

It is well-known that the Dirichlet Laplacian $-\Delta_\Omega$ on $\Omega$ satisfies
\begin{equation*}
\tr\left( -h^2\Delta_\Omega -1 \right)_- \, =  \, L_{1,d}^{(1)} \, |\Omega| \,
h^{-d} - L_{1,d}^{(2)} \, |\partial \Omega|
\, h^{-d+1} + o(h^{-d+1}) \,,
\end{equation*}
see, e.g., \cite{FraGei11a} for a proof under the sole assumption that $\partial \Omega \in C^{1,\alpha}$ for some $0 < \alpha \leq 1$. Here
$$
L_{1,d}^{(1)} \, = \, \frac{1}{(2\pi)^d} \int_{\R^d} \lk |p|^{2}-1 \rk_- \, dp
$$
and, by an argument similar to that in our Proposition \ref{const}, one can bring the second constant in the form
$$
L_{1,d}^{(2)} \, = \, \frac{|\Sph^{d-2}|}{(2\pi)^{d-1}} \frac{2}{(d-1)(d+1)} \, \tr \left[ \chi B^{-(d-1)/2} \chi - (B^+)^{-(d-1)/2} \right]
$$
where $B=-\frac{d^2}{dt^2}+1$ in $L^2(\R)$ and $B^+=-\frac{d^2}{dt^2}+1$ with Dirichlet boundary conditions in $L^2(\R_+)$. A short computation, using the fact that
$$
(E^s-1)_- = s(1-s) \int_0^1 (E-\tau)_- \tau^{s-2} \,d\tau + s (E-1)_- \,,
$$
gives
\begin{align*}
\tr\left( \left(-h^2\Delta_\Omega\right)^s -1 \right)_- \, & =  \, L_{1,d}^{(1)} \, |\Omega| \, h^{-d} \, s \left( (1-s)\int_0^1 \tau^{d/2+s-1}\,d\tau + 1 \right) \\
& \qquad - L_{1,d}^{(2)} \, |\partial \Omega| \, h^{-d+1} \, s \left( (1-s)\int_0^1 \tau^{(d-1)/2+s-1}\,d\tau + 1 \right) + o(h^{-d+1}) \\
& =  \, L_{s,d}^{(1)} \, |\Omega| \, h^{-d} - \frac{s(d+1)}{d-1+2s} \, L_{1,d}^{(2)} \, |\partial \Omega| \, h^{-d+1} + o(h^{-d+1}) \,,
\end{align*}
that is,
$$
\Lzt = \frac{s(d+1)}{d-1+2s} L_{1,d}^{(2)} = \frac{|\Sph^{d-2}|}{(2\pi)^{d-1}} \frac{2s}{(d-1)(d-1+2s)} \, \tr \left[ \chi B^{-(d-1)/2} \chi - (B^+)^{-(d-1)/2} \right] \,.
$$
Since
$$
B^{-(d-1)/2} (t,t) = \frac1{2\pi} \int_\R \frac1{(1+p^2)^{(d-1)/2}} \,dp = A^{-(d-1)/2s}(t,t) 
$$
we find that
$$
\Lzt - \Lz = \frac{|\Sph^{d-2}|}{(2\pi)^{d-1}} \frac{2s}{(d-1)(d-1+2s)} \, \tr \left[ (A^+)^{-(d-1)/2s} - (B^+)^{-(d-1)/2} \right] \,.
$$
We now apply Lemma \ref{song} with $B = -d^2/dt^2+1$ in $L^2(\R)$, with $P$ being the projection onto $L^2(\R_+)$ and with $\phi(E)=E^s$. Then $\phi(PBP)= (B^+)^s$ and $P\phi(B)P=A^+$, and therefore \eqref{eq:song} yields
$$
(B^+)^s \geq A^+ \,.
$$
Since  $E\mapsto E^{-(d-1)/2s}$ is strictly monotone and since the operators $A^+$ and $(B^+)^s$ are not identical, we conclude that
$$
\tr \left[ (A^+)^{-(d-1)/2s} - (B^+)^{-(d-1)/2} \right] > 0 \, .
$$
This shows that $\Lzt - \Lz>0$ and completes the proof of Proposition \ref{comp}.
\qed


\appendix

\section{Equivalence of \eqref{eq:mainintro} and \eqref{eq:mainintroramp}}

For the sake of completeness we include a short proof of

\begin{lemma}
 \label{asymptequiv}
Let $(\lambda_k)_{k \in \N}$ be a non-decreasing sequence of real numbers and
let $A,C > 0$, $B,D \in \R$ and $-1<a-1<b<a$ be related by
$$
C = A^{-1/a} a (a+1)^{-(1+a)/a}\,,
\qquad 
D = B (A(a+1))^{-(1+b)/a} \,.
$$
Then the asymptotic formula
\begin{equation}
\label{eq:summe1}
\sum_{k=1}^N \lambda_k = A N^{a+1} + B N^{b+1} (1+o(1)) \, , \quad N \to \infty
\, ,
\end{equation}
is equivalent to 
\begin{equation}
\label{eq:summe2}
\sum_{k \in \N} (\Lambda - \lambda_k)_+ = C \Lambda^{(1+a)/a} - D
\Lambda^{(1+b)/a} (1+o(1)) \, , \quad \Lambda \to \infty \, .
\end{equation}
\end{lemma}

\begin{proof}
This lemma is a consequence of Hardy, Littlewood and Polya's majorization
theorem, which says that for any non-decreasing sequences $\{a_k\}$ and
$\{b_k\}$
\begin{equation}
 \label{eq:maj}
\sum_{k=1}^N a_k \leq \sum_{k=1}^N b_k
\qquad \text{for all}\ N\in\N
\end{equation}
is equivalent to
$$
\sum_{k=1}^\infty (\Lambda-a_k)_+ \geq \sum_{k=1}^\infty (\Lambda - b_k)_+
\qquad \text{for all}\ \Lambda\in\R \,;
$$
see, e.g., \cite[Prop. 4.B.4]{MarOlk79}. As usual, we will denote property
\eqref{eq:maj} by $\{a_k\}\prec \{b_k\}$.

We fix $\epsilon > 0$ and set $\beta^{\pm}_k =  A (a+1) k^a + (B\pm \epsilon)
(b+1)k^b$.
Note that the assumptions on  $a$ and $b$ imply
\begin{align}\label{eq:discas}
\sum_{k=1}^N \beta^\pm_k = AN^{a+1} + (B\pm \epsilon) N^{b+1} (1+o(1)) \, ,
\quad N \to \infty \,,
\end{align}
and
\begin{align}
\label{eq:contas}
\sum_{k \in \N} (\Lambda - \beta^\pm_k)_+ = \frac{aA}{(A(a+1))^{1+1/a}}
\Lambda^{(1+a)/a} - \frac{B\pm \epsilon}{(A(a+1))^{(1+b)/a}} \Lambda^{(1+b)/a}
(1+o(1)) \, ,
\quad \Lambda \to \infty \,.
\end{align}

First, we assume that \eqref{eq:summe1} holds. Then, by \eqref{eq:summe1} and
\eqref{eq:discas} there is an $N_\epsilon \in \N$ such that for all $N \geq
N_\epsilon$
$$
\sum_{k=1}^N \beta^-_k \leq  \sum_{k=1}^N \lambda_k \leq \sum_{k=1}^N \beta^+_k
\, .
$$
We put $\alpha_k^\pm= \beta_k^\pm$ for $k \geq N_\epsilon$ and $\alpha_k^+ =
\max (\beta^+_k,\lambda_k)$, $\alpha_k^- = \min (\beta^-_k,\lambda_k)$ for $k <
N_\epsilon$. Thus $\{\alpha^-_k\}\prec\{\lambda_k\}\prec\{\alpha^+_k\}$, and
therefore
\begin{equation*}
\sum_{k \in \N} (\Lambda - \alpha^+_k)_+ \geq \sum_{k \in \N} (\Lambda -
\lambda_k)_+  \geq \sum_{k \in \N} (\Lambda - \alpha^-_k)_+ 
\quad\text{for all}\ \Lambda \in \R \,.
\end{equation*}
Since $\sum_{k \in \N} (\Lambda - \alpha^\pm_k)_+  = \sum_{k \in \N} (\Lambda -
\beta^\pm_k)_+ + O(1)$, the assertion \eqref{eq:summe2} follows from
\eqref{eq:contas}. The converse implication is proved similarly.
\end{proof}


\section{The one-dimensional model operator}

Here we outline the calculations that are necessary to complete the analysis of the model operator $A^+$ introduced in Section \ref{sec:half}. The results depend on the following spectral representation of the operator $A^+$ found in \cite{Kwasni10a}.

\begin{theorem}
\label{thm:kwa}
For $E > 0$ let
$$
\psi(E) = (E+1)^s -1
$$
and for $\lambda > 0$ put $\gamma_\lambda(\xi) = 0$ if $0 < \xi < 1$ and
\begin{align*}
\gamma_\lambda(\xi) \, = \, &\frac 1\pi \frac{\lambda \, \psi'(\lambda^2) \, \sin(\pi s) \, (\xi^2-1)^s}{\psi(\lambda^2)^2+(\xi^2-1)^s-2\psi(\lambda^2)(\xi^2-1)\cos(\pi s)} \\
& \times \exp \lk - \frac 1\pi \int_0^\infty \frac{\xi}{\xi^2 + \zeta^2} \ln \frac{\psi'(\lambda^2)(\lambda^2-\zeta^2)}{\psi(\lambda^2)-\psi(\zeta^2)} d\zeta \rk 
\end{align*}
if $\xi \geq 1$. Moreover, define a phase-shift
\begin{equation}
\label{eq:theta}
\vartheta_\lambda \, = \, \frac 1\pi \int_0^\infty \frac{\lambda}{\zeta^2-\lambda^2} \ln \frac{\psi'(\lambda^2)(\lambda^2-\zeta^2)}{\psi(\lambda^2)-\psi(\zeta^2)} \, d\zeta
\end{equation}
and functions
\begin{equation}
\label{eq:ap:eig}
F_\lambda(x) \, = \, \sin \lk \lambda x +\vartheta_\lambda \rk + \int_0^\infty e^{-x\xi} \, \gamma_\lambda (\xi) \, d\xi \, , \quad x > 0 \, .
\end{equation}
Then
$$
\Phi f(\lambda) \, = \, \sqrt{ \frac 2\pi} \int_0^\infty f(x) \, F_\lambda(x) \, dx
$$
defines a unitary operator from $L^2(\R_+)$ to $L^2(\R_+)$.

This operator diagonalizes $A^+$ in the sense that a function $f\in L^2(\R_+)$ is in the domain of $A^+$ if and only if $(\lambda^2+1)^s \Phi f (\lambda)$ is in $L^2(\R_+)$, and in this case
$$
\Phi A^+ f(\lambda) \, = \, (\lambda^2+1)^s \Phi f(\lambda) \, .
$$
\end{theorem}

According to \cite{Kwasni10a} the Laplace transform of $\gamma_\lambda$ is a completely monotone function bounded by one. From \eqref{eq:ap:eig} it follows that for all $t \geq 0$
\begin{equation}
\label{eq:ap:bound}
|F_\lambda(t)| \leq 2 \,  .
\end{equation}

Theorem \ref{thm:kwa} states that the functions $F_\lambda$ are generalized eigenfunctions of the operator $A^+$. Hence, we can write
\begin{equation}
\label{eq:ap:rep}
e^+(t,u,\mu)  \, = \, \frac 2\pi \int_0^\infty \lk (\lambda^2+1)^s-\mu \rk_-^0
F_\lambda(t) F_\lambda(u) \, d\lambda \, .
\end{equation}
{F}rom \eqref{eq:K}, \eqref{eq:mod:lz}, and Proposition \ref{const} it follows that
\begin{equation}
\label{eq:ap:lz}
\Lz \, = \, \frac{4s}{(d-1+2s)(d-1)} \frac{|\Sph^{d-2}|}{(2\pi)^d} \int_0^\infty \int_0^\infty \lk 1-2F_\lambda^2(t) \rk \lk \lambda^2+1\rk^{-(d-1)/2} d\lambda \, dt \, .
\end{equation}


\subsection{Proof of Lemma \ref{eest}}
\label{ap:eest}
Lemma \ref{eest} is an immediate consequence of \eqref{eq:ap:rep}. In view of \eqref{eq:ap:bound} we estimate 
$$
\left| e^+(t,u,\mu) \right| \, \leq \, C \int_0^{(\mu^{1/s}-1)_+^{1/2}} d\lambda \, \leq \, C \mu^{1/(2s)} \, .
$$
This proves the lemma.


\subsection{Proof of Lemma \ref{aest}}
\label{ap:aest}

First we need the following technical result about $\vartheta_\lambda$.

\begin{lemma}
\label{lem:ap:rem1}
The phase-shift $\vartheta_\lambda$ is monotone increasing and twice differentiable in $\lambda > 0$. It satisfies
$$
\vartheta_0 \, = \, 0 \qquad \textnormal{and} \qquad \vartheta_\lambda \to \frac \pi4 (1-s) \quad \textnormal{as} \quad \lambda \to \infty \, .
$$
The first and second derivatives are bounded and one has, as $\lambda \to \infty$,
$$
\frac{d\vartheta_\lambda}{d\lambda} \, = \, \frac{d^2\vartheta_\lambda}{d\lambda^2} \, = \, O \lk \frac 1\lambda \rk  \,.
$$
\end{lemma}

\begin{proof}
Following \cite{Kwasni10a}, we substitute $\zeta = \lambda z$ for $\zeta \in (0,\lambda)$ and $\zeta = \lambda / z$ for $\zeta \in (\lambda,\infty)$ in the definition of $\vartheta_\lambda$ and obtain
$$
\vartheta_\lambda \, = \, \frac 1\pi \int_0^1 \frac{1}{1-z^2} \ln \lk \frac{1}{z^2} \frac{\psi(\lambda^2)- \psi(\lambda^2 z^2)}{\psi(\lambda^2/z^2) - \psi(\lambda^2)} \rk dz.
$$
Note that the function 
$$
\frac{1}{z^2} \frac{\psi(\lambda^2)- \psi(\lambda^2 z^2)}{\psi(\lambda^2/z^2) - \psi(\lambda^2)} \, = \, \frac{1}{z^2} \frac{(1+\lambda^2)^s - (1+\lambda^2 z^2)^s}{(1+\lambda^2/z^2)^s - (1+\lambda^2)^s}
$$
equals $1$ for $\lambda =0$ and that for all $z \in (0,1)$ it is increasing in $\lambda > 0$ and tends to $z^{2s-2}$ as $\lambda$ tends to infinity.
By Lebesgue's dominated convergence we find $\vartheta_0 = 0$ and 
$$
\lim_{\lambda \to \infty} \vartheta_\lambda \, = \, \frac 1\pi \int_0^1 \frac{1}{1-z^2} \ln(z^{2s-2}) \, dz \, = \, \frac \pi4 (1-s) \, .
$$

By \eqref{eq:theta}, we also have
$$
\vartheta_\lambda \, = \, \frac 1\pi \int_0^\infty b_\lambda(\zeta) \, d\zeta \, 
$$
with
$$
b_\lambda(\zeta) \, = \, \frac{\lambda}{\zeta^2-\lambda^2} \, \ln \lk \frac{s(1+\lambda^2)^{s-1} (\lambda^2-\zeta^2)}{(\lambda^2+1)^s-(\zeta^2+1)^s} \rk \, .
$$
We remark that 
$$
\left| \partial_\lambda b_\lambda(\zeta) \right| \, \leq \, \left. \partial_\lambda b_\lambda(\zeta) \right|_{\lambda=0} \, = \, \frac{1}{\zeta^2} \ln \lk \frac{s \zeta^2}{(1+\zeta^2)^s-1} \rk \, .
$$ 
for all $\zeta \in (0,\infty)$.
Since the last expression is integrable in $\zeta \in (0,\infty)$ it follows that
$$
\frac{d\vartheta_\lambda}{d\lambda} \, = \, \frac{1}{\pi} \int_0^\infty \partial_\lambda b_\lambda (\zeta) \, d\zeta
$$
is bounded and, in particular, we obtain
\begin{equation}
\label{eq:ap:dtheta}
\left. \frac{d\vartheta_\lambda}{d\lambda} \right|_{\lambda=0} \, = \, \frac{1}{\pi} \int_0^\infty \frac{1}{\zeta^2} \ln \lk \frac{s \zeta^2}{(1+\zeta^2)^s-1} \rk d\zeta \, .
\end{equation}
Similarly, we can show existence and boundedness of the second derivative and decay of the derivatives  as $\lambda \to \infty$ by explicit calculations and Lebesgue's dominated convergence.
\end{proof}

To simplify notation we put
$$
\psi_\lambda(E) \, = \, \frac{1-E/\lambda^2}{1-\psi(E)/\psi(\lambda^2)}
$$
for $E > 0$. Moreover, we write $G_\lambda$ for the Laplace transform of $\gamma_\lambda$ and $g_\lambda$ for the Laplace transform of $G_\lambda$. According to \cite{Kwasni10a} we have 
\begin{equation}
\label{eq:ap:laplaceg}
g_\lambda(t) \, =  \, \frac{\lambda \cos \vartheta_\lambda + t \sin \vartheta_\lambda}{\lambda^2+t^2} - \lambda^2 \sqrt{\frac{\psi'(\lambda^2)}{\psi(\lambda^2)}} \frac{\varphi_\lambda(t)}{\lambda^2+t^2} \, , \quad t > 0 \, ,
\end{equation}
with
$$
\varphi_\lambda(t) \, = \, \exp \lk \frac 1\pi \int_0^\infty \frac{t}{t^2+\zeta^2} \ln \lk \psi_\lambda(\zeta^2) \rk d\zeta \rk \, .
$$ 

To prove Lemma \ref{aest} we need the following properties of $\varphi_\lambda$.

\begin{lemma}
\label{lem:ap:rem2}
The function $t\mapsto \varphi_\lambda(t)$ is differentiable in $t > 0$ and its
derivative satisfies
\begin{align*}
\varphi'_\lambda(0) \, &= \, o(1) \quad \textnormal{as} \ \lambda \to \infty
\,, \\
\varphi'_\lambda(0) \, &= \, \left. \frac{d\vartheta_\lambda}{d\lambda}
\right|_{\lambda=0} + O(\lambda)  \quad \textnormal{as} \ \lambda \to 0 \,.
\end{align*}
\end{lemma}

\begin{proof}
For fixed $\zeta \in (0,\infty)$ the function $\lambda\mapsto
\psi_\lambda(\zeta^2)$ is non-increasing in $\lambda > 0$ and tends to $1$ as
$\lambda \to \infty$. Moreover,
$$
\frac{1}{\zeta^2}\ln\left( \psi_0(\zeta^2) \right) \, = \, \frac1{\zeta^2}
\ln\left( \frac{s \zeta^2}{(\zeta^2+1)^s-1} \right)
$$
is integrable with respect to $\zeta\in (0,\infty)$. Hence we find that
$$
\varphi'_\lambda(0) \, = \, \frac1\pi \int_0^\infty \frac{1}{\zeta^2} \ln \lk
\psi_\lambda(\zeta^2) \rk d\zeta
$$
and $\varphi'_\lambda(0) = o(1)$ as $\lambda \to \infty$ by Lebesgue's theorem.

In view of (\ref{eq:ap:dtheta})
$$ 
\left. \varphi'_\lambda(0)\right|_{\lambda=0} \, = \, \frac1\pi \int_0^\infty
\frac{1}{\zeta^2} \ln \lk \psi_0(\zeta^2) \rk d\zeta
\, = \, \left. \frac{d\vartheta_\lambda}{d\lambda}
\right|_{\lambda=0} \, .
$$
The second claim now follows from the fact that the derivative of
$\lambda\mapsto\varphi'_\lambda(0)$ is bounded.
\end{proof}

\begin{proof}[Proof of Lemma \ref{aest}]
In view of Theorem \ref{thm:kwa} we can write
$$
a(\mu) - a^+(t,\mu) \, =\, \frac 1\pi \int_0^\infty \lk (\lambda^2+1)^s - \mu
\rk_- \lk 1 - 2 F_\lambda^2(t) \rk d\lambda
$$
and by \eqref{eq:ap:eig} 
$$
1-2F_\lambda(t)^2 \, = \, \cos(2\lambda t + 2 \vartheta_\lambda ) - 4 \sin (\lambda t + \vartheta_\lambda ) \, G_\lambda(t) - 2 G_\lambda(t)^2  \, .
$$
We get
\begin{equation*}
\int_0^\infty  t^\gamma | a(\mu) - a^+(t,\mu) | dt \, \leq  \, R_1(\mu) + R_2(\mu) 
\end{equation*}
with
\begin{align*}
R_1(\mu)  &= \int_0^\infty t^\gamma \left| \int_0^{(\mu^{1/s}-1)_+^{1/2}}  \lk  \mu -(\lambda^2+1)^s \rk  \cos (2\lambda t + 2 \vartheta_\lambda) \, d\lambda \right| dt \, , \\
R_2(\mu)  &= \int_0^\infty t^\gamma \left| \int_0^{(\mu^{1/s}-1)_+^{1/2}}  \lk  \mu -(\lambda^2+1)^s \rk \lk 2 \sin (\lambda t + \vartheta_\lambda ) \, G_\lambda(t) +  G_\lambda(t)^2 \rk  d\lambda \right| dt \, .
\end{align*}

To estimate $R_1(\mu)$ we split the integration in $t$ and integrate over $t \in [0,1]$ first. We assume $0 < \gamma < 1$. The proof for $\gamma = 0$ follows similarly. 

We write 
$$
\cos (2\lambda t + 2 \vartheta_\lambda ) \, = \, \frac {1}{2t} \, \frac{d}{d\lambda} \sin (2\lambda t + 2 \vartheta_\lambda ) - \frac{\cos (2\lambda t + 2 \vartheta_\lambda )}{t} \, \frac{d \vartheta_\lambda }{d\lambda}
$$
and insert this identity in the expression for $R_1(\mu)$. After integrating by parts in the $\lambda$-integral one can estimate
$$
\int_0^1 t^\gamma \left| \int_0^{(\mu^{1/s}-1)_+^{1/2}}  \lk  \mu -(\lambda^2+1)^s \rk  \cos (2\lambda t + 2 \vartheta_\lambda) \, d\lambda \right| dt
\, \leq \, C \mu \lk (\ln \mu)^2 +1 \rk \, .
$$
To estimate the integral over $t \in [1, \infty)$ we proceed similarly. We integrate by parts twice and get
$$
\int_1^\infty t^\gamma \left| \int_0^{(\mu^{1/s}-1)_+^{1/2}}  \lk  \mu -(\lambda^2+1)^s \rk  \cos (2\lambda t + 2 \vartheta_\lambda) \, d\lambda \right| dt \, \leq \, C \mu (\ln \mu +1) \, . 
$$
We conclude
\begin{equation*}
R_1(\mu) \, \leq \, C \mu \lk (\ln \mu)^2 +1 \rk 
\end{equation*}
and turn to estimating $R_2(\mu)$.

Since $G_\lambda$ is non-negative and uniformly bounded, we have
\begin{equation}
\label{eq:ap:r2}
R_2(\mu) \, \leq  \, C \int_0^{(\mu^{1/s}-1)_+^{1/2}} \lk \mu - (\lambda^2+1)^s \rk \int_0^\infty  t^\gamma \, G_\lambda(t)  dt \, d\lambda \, .
\end{equation}
By definition, $g_\lambda(0) = \int_0^\infty G_\lambda(t)  dt$ and $g'_\lambda(0) = \int_0^\infty t  G_\lambda(t) dt$.  We note that, by \eqref{eq:ap:laplaceg},
$$
g_\lambda(0) \, = \, \frac{\cos \vartheta_\lambda}{\lambda} - \sqrt{\frac{\psi'(\lambda^2)}{\psi(\lambda^2)}} 
$$
and apply Lemma \ref{lem:ap:rem1} to estimate $\int_0^\infty G_\lambda(t)  dt  \leq  C \lk \lambda \land \lambda^{-1} \rk$. 
Moreover, by (\ref{eq:ap:laplaceg}),
$$
g'_\lambda(0) \, = \, \frac{\sin \vartheta_\lambda}{\lambda^2} - \sqrt{\frac{\psi'(\lambda^2)}{\psi(\lambda^2)}} \, \varphi_\lambda'(0) 
$$
and we apply Lemma \ref{lem:ap:rem1} and Lemma \ref{lem:ap:rem2} to estimate $\int_0^\infty t   G_\lambda(t) dt   \leq  C \lk 1 \land \lambda^{-1} \rk$. It follows that
$$
\int_0^\infty t^\gamma  \, G_\lambda(t) \, dt  \,  \leq \,  C \lk 1 \land \lambda^{-1} \rk \, .
$$
Thus, by (\ref{eq:ap:r2}), we arrive at
\begin{equation*}
R_2(\mu) \,  \leq \,  C \int_0^{(\mu^{1/s}-1)_+^{1/2}} \lk \mu - (\lambda^2+1)^s \rk \lk 1 \land \lambda^{-1} \rk \, d\lambda \, \leq \, C \, \mu \lk \ln \mu +1 \rk \, . 
\end{equation*}
This finishes the first part of the proof of Lemma \ref{aest}. 

In order to prove the assertion about $K(t)$, we bound
$$
\int_0^\infty t^\gamma \, |K(t)| \, dt \leq \int_{|\xi'|<1}
|\xi'|^{1+2s} \int_0^\infty t^\gamma | a^+(t|\xi'|,|\xi'|^{-2s}) -
a(|\xi'|^{-2s}) |\,dt \,d\xi' \,.
$$
Here we also used that, since $a(\mu)=a^+(t,\mu)=0$ for $\mu\leq 1$, we can restrict the integration in the definition of $K$ to $|\xi'|<1$. On the other hand, from \eqref{eq:aest} we know that
$$
\int_0^\infty t^\gamma | a^+(t\mu^{-1/2s},\mu) -a(\mu) |\,dt \leq C_\gamma
\mu^{1+(\gamma+1)/(2s)} \lk (\ln\mu)^2 +1 \rk \,.
$$
Combining these two bounds and using that $\gamma<1\leq d-1$ we obtain the
second part of Lemma \ref{aest}.
\end{proof}


%

\subsection{A remainder estimate}
\label{ap:B2}

The following technical lemma was needed in the proof of the upper bound near
the boundary.

\begin{lemma}\label{lem:remest}
 Assume that $\phi \in C_0^1(\R^d)$ is supported in a ball of radius $l =1$ and
that (\ref{eq:int:gradphi}) is satisfied with $l=1$. Then for any $\frac 12 -s <
\sigma < \min\{\frac 12, 1-s\}$ one has
\begin{equation}
\label{derivatives}
\int_{\R^d} \int_{\R^d} \left| ( -\Delta_{x'})^\sigma \frac{|\phi(x) -
\phi(y)|^2}{|x-y|^{d+2s}} \right| dx dy \, \leq \, C
\end{equation}
\end{lemma}

\begin{proof}
For  $x = (x',x_d) \in \R^{d-1} \times \R$ and $y = (y',y_d) \in \R^{d-1} \times
\R$ put
$$
F_{x_d,y}(x') \, = \, \frac{\lk \phi(x',x_d) - \phi(y',y_d)\rk^2}{\lk |x'-y'|^2 + (x_d-y_d)^2\rk^{d/2+s}} \, .
$$
To establish (\ref{derivatives}) we use the fact that
\begin{equation}
\label{der}
\left| ( -\Delta_{x'})^\sigma  \frac{|\phi(x) - \phi(y)|^2}{|x-y|^{d+2s}} \right| \, \leq \, C \int_{\R^{d-1}}  \frac{| F_{x_d,y}(x') - F_{x_d,y}(z')|}{|x'-z'|^{d-1+2\sigma}} \, dz'
\end{equation}
and split the integration in $x \in \R^d$ and $y \in \R^d$ in four parts. First we assume that $x$ and $y$ are in $B_1$. Then we have to show that
\begin{align}
\nonumber
& \int_{B_1} \int_{B_1} \int_{\R^{d-1}}  \frac{| F_{x_d,y}(x') - F_{x_d,y}(z')|}{|x'-z'|^{d-1+2\sigma}} \, dz' \, dx \, dy\, =  \\
\nonumber
& \, \int_{B_1} \int_{B_1} \int_{|x'-z'|<|x-y|/2}  \frac{| F_{x_d,y}(x') - F_{x_d,y}(z')|}{|x'-z'|^{d-1+2\sigma}} \, dz' \, dx \, dy \\
\label{fint}
& + \int_{B_1} \int_{B_1} \int_{|x'-z'|\geq |x-y|/2}  \frac{| F_{x_d,y}(x') - F_{x_d,y}(z')|}{|x'-z'|^{d-1+2\sigma}} \, dz' \, dx \, dy
\end{align}
is bounded from above.

To estimate the first integral over $|x'-z'|<|x-y|/2$ we use the fact that
$$
F(z')-F(x') \, = \, \sum_{j=1}^{d-1} \frac{(z_j-x_j)}{|x'-z'|} \int_0^{|x'-z'|} (\partial_j F) \lk x'+t \frac{(z'-x')}{|x'-z'|} \rk \, dt \, .
$$
For $j = 1, \dots, d-1$ we have
$$
(\partial_j F_{x_d,y})(x') \, = \, \frac{2(\phi(x',x_d)-\phi(y)) (\partial_j \phi(x))}{|x-y|^{d+2s}} - (d+2s) (x_j-y_j) \frac{(\phi(x)-\phi(y))^2}{|x-y|^{d+2s+2}} \, ,
$$
thus 
$$
|(\partial_j F_{x_d,y})(x') | \leq C \, |x-y|^{-d+1-2s} \, .
$$
Hence, we obtain
\begin{align}
\nonumber
& |F_{x_d,y}(z')-F_{x_d,y}(x')| \\
\label{fdiff}
& \leq C |x'-z'|^\alpha \lk  \int_0^{|x'-z'|} \lk \left| x'+t \frac{(z'-x')}{|x'-z'|} - y' \right|^2 + (x_d-y_d)^2 \rk^{\beta} dt \rk^{1-\alpha} \, ,
\end{align}
with $0 < \alpha < 1$ and $\beta = (\frac{d-1}{2}+s)/(\alpha-1)$, by applying H\"older's inequality.
Note that
\begin{align*}
\left| x'-y'+t \frac{(z'-x')}{|x'-z'|} \right|^2 +(x_d-y_d)^2 &=  |x-y|^2 + t^2 + 2t \frac{(x'-y') \cdot (z'-x')}{|x'-z'|} \\
&\geq \lk |x-y|-t \rk^2 \, .
\end{align*}
Inserting this into (\ref{fdiff}) we get for $|x'-z'| < |x-y|/2$
\begin{align*}
|F_{x_d,y}(z')-F_{x_d,y}(x')|  &\leq  C |x'-z'|^\alpha \lk  \int_0^{|x-y|/2}  (|x-y|-t)^{2\beta} dt \rk^{1-\alpha} \\
&\leq C |x'-z'|^\alpha |x-y|^{(2\beta+1)(1-\alpha)} \, ,
\end{align*}
where $(2\beta+1)(1-\alpha) = -d-2s+2-\alpha$. We conclude that for any $2\sigma < \alpha < 1$ and $\sigma < 1-s$
\begin{align}
\nonumber
& \int_{B_1} \int_{B_1} \int_{|x'-z'|<|x-y|/2}  \frac{| F_{x_d,y}(x') - F_{x_d,y}(z')|}{|x'-z'|^{d-1+2\sigma}} \, dz' \, dx \, dy \\
\nonumber
& \leq C \int_{B_1} \int_{B_1} \int_{|x'-z'|<|x-y|/2} |x'-z'|^{-d+1-2\sigma+\alpha} dz' \, |x-y|^{-d-2s+2-\alpha} \,  dx \, dy \\
\label{firstint}
&\leq C \, .
\end{align}

Now we turn to the second integral in (\ref{fint}) over $|x'-z'|\geq|x-y|/2$. Since 
\begin{equation}
\label{fest}
0 \leq F_{x_d,y}(x') \leq |x-y|^{-d-2s+2}
\end{equation}
and $\sigma < 1-s$ we have
\begin{equation}
\label{secondint1}
\int_{B_1} \int_{B_1} \int_{|x'-z'|\geq |x-y|/2} \frac{F_{x_d,y}(x')}{|x'-z'|^{d-1+2\sigma}} dz'  dx  dy 
\leq C  \int_{B_1} \int_{B_1} \frac{1}{|x-y|^{d+2s-2+2\sigma}}  \leq  C  .
\end{equation}
Moreover,
$$
\int_{|x'-z'|\geq |x-y|/2} \frac{F_{x_d,y}(z')}{|x'-z'|^{d-1+2\sigma}}  dz'  \leq  C  |x-y|^{-d+1-2\sigma+(d-1)/p} \lk \int_{|x'-z'|\geq |x-y|/2} F_{x_d,y}^q(z')  dz' \rk^{1/q}  
$$
with $\frac 1p + \frac 1q = 1$, by H\"older's inequality. Since $\sigma > \frac 12 -s$  we can choose $p > \frac{d-1}{2\sigma}$ and $q > \frac{d-1}{d+2s-2}$. By (\ref{fest}), we have
\begin{align*}
\lk \int_{|x'-z'|\geq |x-y|/2} F_{x_d,y}^q(z')  dz' \rk^{1/q} \, & \leq \, C \lk \int_{\R^{d-1}} \lk |z'-y'|^2 + (x_d-y_d)^2 \rk^{-q(d/2+s-1)} dz' \rk^{1/q} \\
& \leq C \, |x_d-y_d|^{-d-2s+2+(d-1)/q} \, .
\end{align*}
It follows that
\begin{align*}
&\int_{B_1} \int_{B_1} \int_{|x'-z'|\geq |x-y|/2} \frac{F_{x_d,y}(z')}{|x'-z'|^{d-1+2\sigma}} \, dz'  \, dx \, dy \\
& \leq \, C \int_{B_1} \int_{B_1}  |x-y|^{-d+1-2\sigma+(d-1)/p} \, |x_d-y_d|^{-d-2s+2+(d-1)/q} \, dx \, dy \\
& \leq \, C \int_0^2 t^{-d-2s+2+(d-1)/q} \int_0^2 r^{d-2} \lk  r^2 + t^2 \rk^{(-d+1-2\sigma)/2+(d-1)/(2p)} dr \, dt \, ,
\end{align*}
where we substituted $t = |x_d-y_d|$ and $r = |x'-y'|$. Since $p > \frac{d-1}{2\sigma}$ and $\sigma < 1-s$ we find
\begin{equation}
\label{secondint2}
\int_{B_1} \int_{B_1} \int_{|x'-z'|\geq |x-y|/2} \frac{F_{x_d,y}(z')}{|x'-z'|^{d-1+2\sigma}} \, dz'  \, dx \, dy \, \leq \, C \int_0^2 t^{1-2s-2\sigma} dt \, \leq \, C \, .
\end{equation}
The estimates (\ref{secondint1}) and (\ref{secondint2}) show that 
\begin{equation}
\label{secondint}
\int_{B_1} \int_{B_1} \int_{|x'-z'|\geq |x-y|/2}  \frac{| F_{x_d,y}(x') - F_{x_d,y}(z')|}{|x'-z'|^{d-1+2\sigma}} \, dz' \, dx \, dy  \, \leq \, C 
\end{equation}
and from (\ref{der}), (\ref{firstint}), and (\ref{secondint}) it follows that
$$
\int_{B_1} \int_{B_1}  \left| ( -\Delta_{x'})^\sigma \frac{|\phi(x) - \phi(y)|^2}{|x-y|^{d+2s}} \right| dx dy \, \leq \, C \, .
$$

The proof that the respective integrals over $B_1 \times (\R^d\setminus B_1)$, $(\R^d \setminus B_1) \times B_1$, and $(\R^d \setminus B_1) \times (\R^d \setminus B_1)$ are finite is similar but easier, since $\textnormal{supp} \phi \subset B_1$ and we only have to handle one singularity at a time. 
\end{proof}


\subsection*{Acknowledgements}
The authors are grateful to R. Ba\~nuelos, M. Kwa\'snicki and B. Siudeja for helpful correspondence and to the anonymous referee for his/her help to improve the paper. This work is partially supported by NSF grants PHY-1068285 (R.L.F.) and PHY-1122309 (L.G.) and DFG grant GE 2369/1-1 (L.G.).


\end{document}